\documentclass[11pt,reqno,a4paper]{amsart}
\usepackage{graphicx}

\usepackage{amssymb}
\usepackage{amsthm}
\usepackage{amsmath}

\usepackage[hdivide={2.5cm,,2.5cm}, vdivide={3.5cm,,2.8cm}]{geometry} 

\newtheorem{thm}{Theorem}[section]
\newtheorem{cor}[thm]{Corollary}
\newtheorem{lem}[thm]{Lemma}
\newtheorem{prop}[thm]{Proposition}

\theoremstyle{definition}
\newtheorem{definition}[thm]{Definition}

\newtheorem{rem}[thm]{Remark}
\newtheorem{ex}[thm]{Example}

\newcommand{\CC}{\widehat{\mathbb{C}}}
\newcommand{\C}{\mathbb{C}}
\newcommand{\D}{\mathbb{D}}
\newcommand{\TT}{\mathcal{T}}
\renewcommand{\H}{\mathbb{H}}

\newcommand{\N}{\mathbb{N}}
\newcommand{\R}{\mathbb{R}}
\newcommand{\QQ}{\mathcal{Q}}

\newcommand{\hol}{\textup{Hol}}
\newcommand{\A}{\mathcal{A}}
\renewcommand{\S}{\mathcal{S}}
\newcommand{\B}{\mathcal{B}}
\newcommand{\CTC}{\mathcal{C}}
\newcommand{\K}{\mathcal{K}}

\newcommand{\F}{\mathcal{F}}

\newcommand{\LC}{\mathtt{LC}}
\newcommand{\EF}{\mathtt{EF}}
\newcommand{\HV}{\mathtt{HV}}
\newcommand{\HF}{\mathtt{HF}}
\newcommand{\BP}{\mathtt{BP}}
\newcommand{\DW}{\mathtt{DW}}
\newcommand{\LL}{\mathcal{L}}%

\newcommand{\no}{\noindent}
\newcommand{\dstyle}{\displaystyle}
\renewcommand{\Re}{\textup{Re}\,}
\renewcommand{\Im}{\textup{Im}\,}
\newcommand{\closure}{\overline}

\newcommand{\de}{\partial}
\DeclareMathOperator{\esup}{ess\sup}

\makeatletter
    
    \@addtoreset{equation}{section}
  \makeatother

\title[Loewner theory for quasiconformal extensions: old and new]{Loewner theory for quasiconformal extensions:\\ old and new}
\author[I. Hotta]{Ikkei Hotta}
\address{Department of Applied Science, Yamaguchi University, 2-16-1 Tokiwadai, Ube 755-8611, Japan}
\email{ihotta@yamaguchi-u.ac.jp}
\thanks{This work was supported by Grant-in-Aid for JSPS Fellows (13J02250) and Grant-in-Aid for Young Scientists (B) (26800053)\\\indent This paper will be included in the proceedings of the 2nd GSIS-RCPAM International Symposium ``Geometric Function Theory and Applications in Sendai'' which was held in Tohoku University on September 10th-13th, 2013.}

\begin{document}

\begin{abstract}
This survey article gives an account of quasiconformal extensions of univalent functions with its motivational background from Teichm\"uller theory and classical and modern approaches based on Loewner theory.
\end{abstract}

\maketitle

\

\tableofcontents

\newpage

	%
			\section{\bf Universal Teichm\"uller spaces}


The notion of the universal Teichm\"uller spaces was illuminated in the theory of quasiconformal mappings as an embedding of the Teichm\"uller spaces of compact Riemann surfaces of finite genus.
Several equivalent models of universal Teichm\"uller spaces are known (see e.g. \cite{Sugawa:2007}).
In this article we will focus on the connection with a space of the Schwarzian derivatives of conformal extensions of quasiconformal mappings defined on the upper half-plane $\H^+ := \{ z \in \C : \Im z > 0\}$.

	%
			\subsection{Quasiconformal mappings}
	%

A homeomorphism $f$ of a domain $G \subset \C$ is called \textbf{\textit{k}-quasiconformal} if $f_{z}$ and $f_{\bar{z}}$, the partial derivatives in $z$ and $\bar{z}$ in the distributional sense, are locally integrable on $G$ and satisfy 
\begin{equation}
\label{qc_def}
|f_{\bar{z}}(z)| \leq k |f_{z}(z)|
\end{equation}
almost everywhere in $G$, where $k \in [0,1)$. 
The above definition implies that a quasiconformal map $f$ is sense-preserving, namely, the Jacobian $J_{f} := |f_{z}|^{2} - |f_{\bar{z}}|^{2}$ is always positive.

In order to observe the geometric interpretation of the inequality \eqref{qc_def}, assume for a while $f \in C^{1}(G)$.
Then the differential is
$$
df = \frac{\de f}{\de z} dz + \frac{\de f}{\de \bar{z}} d \bar{z}.
$$
We shall consider the $\R$-linear transformation $T(z) := f_{z} z + f_{\bar{z}} \bar{z}$.
Denote $\alpha := \arg  f_{z}$ and $\beta := \arg f_{\bar{z}}$.
Then we have
\begin{eqnarray*}
T(re^{i\theta}) &=& |f_{z}|e^{i\alpha} re^{i\theta} + |f_{\bar{z}}|e^{i\beta} re^{-i\theta}\\
		&=& re^{i\psi} \left( |f_{z}|e^{i(\theta -\phi)}+  |f_{\bar{z}}|e^{-i(\theta-\phi)}\right)
\end{eqnarray*}
where $\phi := (\beta - \alpha)/2$ and $\psi := (\beta + \alpha)/2$.
Consequently, $f$ maps each infinitesimal circle in $G$ onto an infinitesimal ellipse with axis ratio bounded by $(1+k)/(1-k)$ (Fig. 1).
\begin{figure}[h]
\begin{center}
\hspace{30pt}\includegraphics[width=380pt]{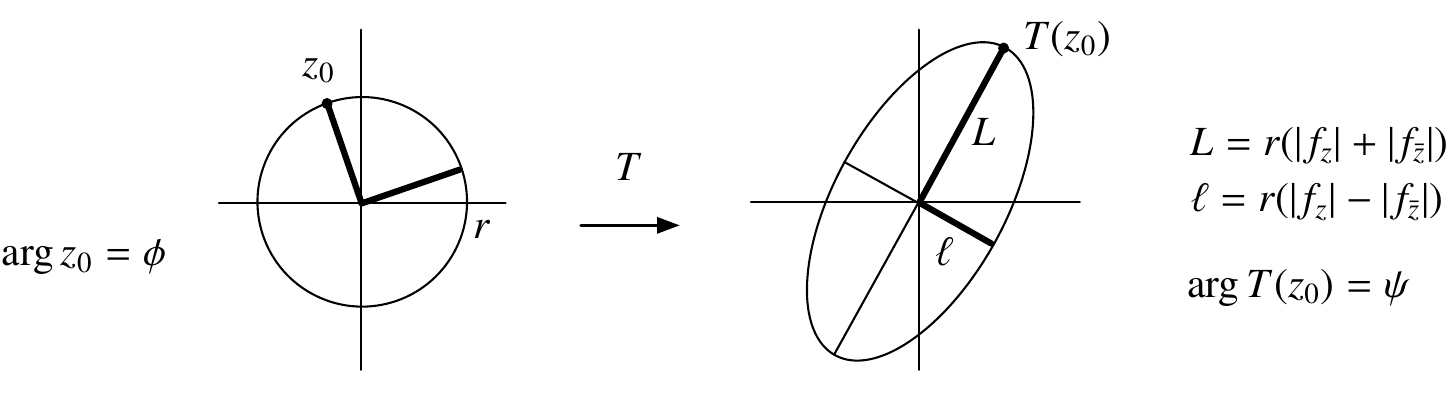}\\
{Figure 1. An infinitesimal circle is mapped to an infinitesimal ellipse.}
\end{center}
\end{figure}

Suppose that $f$ is conformal (i.e., holomorphic injective) on $G$.
Recall that $f_{z} = (f_{x} -if_{y})/2$ and $f_{\bar{z}} = (f_{x} + i f_{y})/2$. 
Thus $f_{\bar{z}}(z) =0$ for all $z \in G$ which is exactly the Cauchy-Riemann equations. 
Hence $f$ is 0-quasiconformal. In this case $L = \ell$ in Figure 1.
Conversely, if $f$ is 0-quasiconformal, then by \eqref{qc_def} $f_{\bar{z}}(z) =0$ for almost all $z \in G$.
By virtue of the following Weyl's lemma, we conclude that $f$ is conformal on $G$.
 \begin{lem}[Weyl's lemma (see e.g. {\cite[p.84]{ImayoshiTaniguchi:1992}})]
Let $f$ be a continuous function on $G$ whose distributional derivative $f_{\bar{z}}$ is locally integrable on $G$. 
If $f_{\bar{z}} =0$ in the sense of distributions on $G$, then $f$ is holomorphic on $G$.
\end{lem}

Let $B(G)$ be the open unit ball $\{\mu \in L^{\infty}(G) : ||\mu||_\infty < 1\}$ of $L^{\infty}(G)$, where $L^{\infty}(G)$ is the complex Banach space of all bounded measurable functions on $G$, and 
$
||\mu||_{\infty} := \esup\displaylimits_{{z \in G}} |\mu(z)|
$
for a $\mu \in L^{\infty}(G)$.
An element $\mu \in B(G)$ is called the \textbf{Beltrami coefficient}.
If $f$ is a $k$-quasiconformal mapping on $G$, then it is verified that $f_z(z) \neq 0 $ for almost all $z \in G$ (e.g. \cite[Theorem IV-1.4 in p.166]{LehtoVirtanen:1973}).
Hence $\mu_f := f_{\bar{z}} / f_z$ defines a function belongs to $B(G)$.
$\mu_f$ is called the \textbf{complex dilatation} of $f$, and the quantity $k := k(f) := ||\mu_f||_{\infty}$ is called the \textbf{maximal dilatation} of $f$.
Conversely, the following fundamental existence and uniqueness theorem is known.

 \begin{thm}[The measurable Riemann mapping theorem]
\label{meas_RMT}
For a given measurable function $\mu \in B(\C)$, there exists a unique solution $f$ of the equation
\begin{equation}
\label{Belt_Eq}
f_{\bar{z}} = \mu f_z
\end{equation}
for which $f : \C \to \C$ is a quasiconformal mapping fixing the points 0 and 1. 
\end{thm}

\no
The equation \eqref{Belt_Eq} is called the \textbf{Beltrami equation}.

Here we give some fundamental properties of quasiconformal mappings we will use later.
For the general theory of quasiconformal mappings in the plane, the reader is referred to \cite{Ahlfors:2006}, \cite{LehtoVirtanen:1973}, \cite{AstaraIwa:2009}, \cite{Hubbard:2006} and \cite{ImayoshiTaniguchi:1992}.

$f$ is $0$-quasiconformal if and only if $f$ is conformal, as discussed above. 
If $f$ is $k$-quasiconformal, then so is its inverse $f^{-1}$ as well.
A composition of a $k_1$- and $k_2$-quasiconformal map is $(k_1 + k_2)/(1 + k_1k_2)$-quasiconformal.
The composition property of the complex dilatation is the following; 
Let $f$ and $g$ be quasiconformal maps on $G$.
Then the complex dilatation $\mu_{g \circ f^{-1}}$ of the map $g \circ f^{-1}$ is given by
\begin{equation}
\label{Belt_formula}
\mu_{g \circ f^{-1}}(f) = \frac{f_z}{\closure{f_z}}\cdot\frac{\mu_g -\mu_f}{1-\mu_g \closure{\mu_f}}.
\end{equation}
Since a $0$-quasiconformal map is conformal, the above formula concludes that if $\mu_f = \mu_g$ almost everywhere in $G$ then $g \circ f^{-1}$ is conformal on $f(G)$.

As the case of conformal mappings, isolated boundary points of a domain $G$ are removable singularities of every quasiconformal mapping of $G$.
It follows from this property that quasiconformal and conformal mappings divide simply connected domains into the same equivalence classes.

	%
			\subsection{Schwarzian derivatives}
	%

Let $f$ be a non-constant meromorphic function with $f' \neq 0$. 
Then we define the \textbf{Schwarzian derivative} by means of
\begin{equation*}
S_f 
	:= \left(\frac{f''}{f'}\right)' -\frac12 \left(\frac{f''}{f'}\right)^2 
	\,\,=\,\, \frac{f'''}{f'} -\frac32 \left(\frac{f''}{f'}\right)^2.
\end{equation*}
It is known that $f$ is a M\"obius transformation if and only if $S_f \equiv 0$.
Further, a direct calculation shows that
$$
S_{f \circ g} = (S_f \circ g)g'^2 + S_g.
$$
Hence it follows the invariance property of $S_f$ that if $f$ is a M\"obius transformation then $S_{f \circ g} = S_g$.
One can interpret that the Schwarzian derivative measures the deviation of $f$ from M\"obius transformations.
In order to describe it precisely, we introduce the norm of the Schwarzian derivative $||S_f||_G$ of a function $f$ on $G$ by
$$
||S_f||_G := \sup_{z \in G} |S_{f}(z)| \eta_G(z)^{-2},
$$
where $\eta_G$ is a Poincar\'e density of $G$.
One of the important properties of $||S_f||$ is the following; 
Let $f$ be meromorphic on $G$ and $g$ and $h$ M\"obius transformations, then $||S_f||_G = ||S_{h \circ f \circ g}||_{g^{-1}(G)}$. 
It shows that $||S_f||$ is completely invariant under compositions of M\"obius transformations.
We note that if $G = \D := \{z \in \C : |z| < 1\}$, then $||S_f||_{\D} = \sup_{z\in\D}(1-|z|)^2|S_f(z)|$. 
For later use, we denote $||S_f||_{\D}$ by simply $||S_f||$.

	%
			\subsection{Bers embedding of Teichm\"uller spaces}
	%

Let us consider the family $\F$ of all quasiconformal automorphisms of the upper half-plane $\H^+$.
Since all mappings in $\F$ can be extended to homeomorphic self-mappings of the closure of $\H^+$, all elements of $\F$ are recognized as self-homeomorphisms of $\closure{\H^{+}}$.
We define an equivalence relation $\sim$ on $\F$ according to which $f \sim g$ for $f, g \in \F$ if and only if there exists a holomorphic automorphism $M$ of $\H^+$, a M\"obius transformation having the form $M(z) = (a z + b) /(cz+d),\, a,b,c,d \in \R$, such that $f \circ M = g$ on $\H^{+}$.
The equivalence relation on $\F$ induces the quotient space $\F/\sim$, which is called the \textbf{universal Teichm\"uller space} and denoted by $\TT$.
Theorem \ref{meas_RMT} with \eqref{Belt_formula} tells us that there is a one-to-one correspondence between $\TT$ and $B(\H^+)$.
If $f \sim g$, then the corresponding complex dilatations $\mu_f$ and $\nu_g$ are also said to be \textbf{equivalent}.

Another equivalent class of $\TT$ is given by the following profound observation due to Bers \cite{Bers:1960}.
Let $\mu \in B(\H^+)$. 
We extend $\mu$ to the lower half-plane $\H^- := \{ z \in \C : \Im z < 0\}$ by putting 0 everywhere.
By Theorem \ref{meas_RMT}, there exists a quasiconformal mapping $f^{\mu}$ fixing 0, 1, $\infty$ associated with such an extended $\mu$.
Then $f^{\mu}|_{\H^-}$ is conformal.

 \begin{thm}[see e.g. {\cite[Theorem III-1.2]{Lehto:1987}}]
The complex dilatations $\mu$ and $\nu$ are equivalent if and only if $f^{\mu}|_{\H^-} \equiv f^{\nu}|_{\H^-}$.
\end{thm} 

\no
By the above theorem, the universal Teichm\"uller space $\TT$ can be understood as the set of the normalized conformal mappings $f^{\mu}|_{\H^-}$ which can be extended quasiconformally to the upper half-plane $\H^+$.
Recall that for a M\"obius transformation $f$ we have $S_{f \circ g} = S_g$.
Therefore, it is natural to consider the mapping
\begin{equation}
\label{mapping_TQ}
\TT \ni [f] \mapsto S_{f^{\mu}|_{\H^-}} \in \QQ,
\end{equation}
between $\TT$ and $\QQ$, where $\QQ$ is the space of functions $\phi$ holomorphic in $\H^-$ for which the hyperbolic sup norm $||\phi||_{\H^{-}} = \sup_{z \in \H^-} 4(\Im z)^2|\phi(z)|$ is finite.

In order to investigate a detailed property of the mapping \eqref{mapping_TQ}, we define a metric on $\TT$ by
\begin{equation*}
d_t(p,q) 
	:= \frac12 \inf_{p, q \in \TT} \{\log K(g \circ f^{-1}) : f \in p,\,g \in q\},
\end{equation*}
where $K(f) := (1 + k(f))/(1-k(f))$, and on $\QQ$ by
\begin{equation*}
d_q (\varphi_1, \varphi_2) := ||\varphi_1 - \varphi_2||_{\infty}.
\end{equation*}
$d_t$ is called the \textbf{Teichm\"uller distance}.
As a consequence of the fact that $d_t$ and $d_q$ are topologically equivalent, we obtain the following theorem which provides a new model of the universal Teichm\"uller space.
 \begin{thm}
The mapping \eqref{mapping_TQ} is a homeomorphism of the universal Teichm\"uller space $\TT$ onto its image in $\QQ$.
\end{thm} 
\no
The mapping \eqref{mapping_TQ} is called the \textbf{Bers embedding} of Teichm\"uller space.
We denote the image of $\TT$ under \eqref{mapping_TQ} by $\TT_1$.
It is known that $\TT_1$ is a bounded, connected and open subset of $\QQ$ (\cite{Ahlfors:1963}).

From the viewpoint of the theory of univalent functions, $\TT_1$ is characterized as follows.
Let $\A$ be the family of functions $f$ holomorphic in $\D$ with $f(0)=0$ and $f'(0)=1$ and $\S$ be the subfamily of $\A$ whose components are univalent on $\D$.
We define $\S(k)$ and $\S^*(k)$ as the families of functions in $\S$ which can be extended to $k$-quasiconformal mappings of $\C$ and $\CC$.
Set $\S(1) := \cup_{k \in [0,1)}\S(k)$.
Then $\TT_1$ is written by
$$
\TT_1 = \{ S_f : f \in \S(1)\}.
$$

We give a short account of the relation to the Teichm\"uller spaces.
Let $S_1$ and $S_2$ be Riemann surfaces and $G_1$ and $G_2$ the covering groups of $\H$ over $S_1$ and $S_2$, respectively.
For the Riemann surfaces $S_1$ and $S_2$, the Teichm\"uller spaces $\TT_{S_1}$ and $\TT_{S_2}$ are defined.
If $G_1$ is a subgroup of $G_2$, then the relation $\TT_{S_2} \subset \TT_{S_1}$ holds.
In particular, if $G_1$ is trivial, then $T_{S_1}$ is the universal Teichm\"uller space which includes all the other Teichm\"uller spaces as subspaces.
For this reason the name ``universal" is used to $\TT_1$.

\

	%
			\section{\bf Quasiconformal extensions of univalent functions}
	%

In Section 1 we have introduced $\S(k)$ to characterize the universal Teichm\"uller space $\TT_1$.
Before entering the main part concerning with Loewner theory, we present some results of the general study of univalent functions and quasiconformal extensions.

	%
			\subsection{Univalent functions}
	%

First of all, we review some results for the class $\S$. 
A number of properties for this class have been investigated by elementary methods.

$\Sigma$, the family of univalent holomorphic maps $g(\zeta) = \zeta + \sum_{n=0}^{\infty}b_n \zeta^{-n}$ mapping $\D^* :=\CC\,\backslash\, \closure{\D}$ into $\CC \,\backslash\,\{0\}$, also plays a key role in the theory of univalent functions.
For $f(z) = z + \sum_{n=2}^{\infty}a_{n}z^{n} \in \S$, define 
\begin{equation*}
\label{ftog}
g(\zeta) := \frac1{f(1/\zeta)} = \zeta - a_{2} + \frac{a_{2}^{2} -a_{3}}{\zeta} + \cdots \hspace{15pt}(\zeta \in \D^{*}).
\end{equation*} 
Then $g \in \Sigma$.
On the other hand it is not always true that for a given $g \in \Sigma$, $f(z) := 1/g(1/z) \in \S$ because $g$ may take 0.
Hence $\Sigma_{0} := \{g \in \Sigma : \textup{$g(\zeta) \neq 0$ on $\zeta \in \D^{*}$}\}$ and $\S$ have a one-to-one correspondence.

 \begin{thm}[Gronwall's area theorem]
For a $g \in \Sigma$, we have
$$
m(\C - g(\D^{*})) = \pi\left(1- \sum_{n=1}^{\infty}n|b_{n}|^{2}\right),
$$
where $m$ stands for the Lebesgue measure.
\end{thm} 

\no
In particular $\sum_{n=1}^{\infty}n|b_{n}|^{2} \leq 1$. In particular $|b_{1}| \leq 1$. Here the equality $|b_{1}| =1$ holds if and only if $g(\zeta) = \zeta + b_{0} + e^{i\theta}/\zeta$.
Since $(f(z^{n}))^{1/n} \in \S$ for all $f \in \S$ and all $n \in \N$, we have

$$
\frac{1}{\sqrt{f(1/\zeta^{2})}} = \zeta - \frac{1}{2}a_{2} \cdot \frac{1}{\zeta} + \cdots \in \Sigma \hspace{15pt}(\zeta \in \D^{*}).
$$
Hence by the estimate for $|b_{1}|$, we obtain the following.

 \begin{thm}[\cite{Bieberbach:1916}]
\label{bie-a2}
If $f \in \S$, then $|a_{2}| \leq 2$.
Equality holds if and only if $f(z)$ is a rotation of the Koebe function defined by
\begin{equation}
\label{Koebef}
K(z) := \frac{z}{(1-z)^{2}} = \frac14\left(\left(\frac{1+z}{1-z}\right)^{2}-1\right) = z + \sum_{n=2}^{\infty}nz_{n}.
\end{equation}
\end{thm} 
Then, in a footnote of the paper \cite{Bieberbach:1916} Bieberbach wrote ``\textit{Vielleicht ist \"uberhaupt $k_{n} = n$}'' (where $k_{n} := \max_{f \in \S}|a_{n}|$) which means that \textit{probably $k_{n} =n$ in general}. 
This statement is called the \textbf{Bieberbach conjecture}.
In 1923 L\"owner \cite{Loewner:1923} proved $|a_{3}| \leq 3$, 
in 1955 Garabedian and Schiffer \cite{GarabedianSchiffer:1955} proved $|a_{4}| \leq 4$, 
in 1969 Ozawa\cite{Ozawa:1969b, Ozawa:1969a} and 
in 1968 Pederson \cite{Pederson:1968} proved $|a_{6}| \leq 6$ independently and 
in 1972 Pederson and Schiffer \cite{PedersonSchiffer:1972} proved $|a_{5}| \leq 5$.
Finally, in 1985 de Branges \cite{deBranges:1985} proved $|a_{n}| \leq n$ for all $n$.
For the historical development of the conjecture, see e.g. \cite{Zorn:1986} and \cite{Koepf:2007}.
The coefficient problem for $\Sigma$  appears to be even more difficult than for $\S$. 
One reason is that there can be no single extremal function for all coefficients as the Koebe function.
In 1914, Gronwall \cite{Gronwall:1914} proved $|b_{1}| \leq 1$ as above,
in 1938 Schiffer \cite{Schiffer:1938} proved $|b_{2}| \le 2/3$ and
in 1955 Garabedian and Schiffer \cite{GarabedianSchiffer:1955b} proved $|b_{3}| \le 1/2 + \exp(-6)$.
We do not even have a general coefficient conjecture for $\Sigma$.
For this problem, see e.g. \cite{NehariNetanyahu:1957}, \cite{SchoberWilliams:1984} and \cite{CarlesonJones:1992}.

We get back to the main story. 
The next is an important application of Theorem \ref{bie-a2}.

 \begin{thm}[The Koebe 1/4-theorem]
\label{1/4-theorem}
If $f \in \S$, then $f(\D)$ contains the disk centered at the origin with radius $1/4$.
\end{thm} 

Since the class $\S$ is closed with respect to the Koebe transform
\begin{equation}
\label{Koebe_trans}
f_K(z) := \frac{f(\frac{z + \zeta}{1+\bar{\zeta}z}) -f(\zeta)}{(1-|\zeta|^2)f'(\zeta)} = z + \left(\frac12 (1-|\zeta|^2)\frac{f''(\zeta)}{f'(\zeta) }-\bar{\zeta}\right)z^2 + \cdots,
\end{equation}
applying Theorem \ref{bie-a2} to \eqref{Koebe_trans} we have the inequality
$$
\left|
(1-|z|^{2}) \frac{f''(z)}{f'(z)} -2\bar{z}
\right|
\leq 4.
$$
It derives the distortion theorems for the class $\S$.

 \begin{thm}
\label{kdistortion}
If $f \in \S$, then 
$$
\begin{array}{ccc}
\dstyle \left|\frac{zf''(z)}{f'(z)} - \frac{2|z|^{2}}{1-|z|^{2}}\right| \le \frac{4|z|}{1-|z|^{2}},\\[12pt]
\dstyle \frac{1-|z|}{(1+|z|)^{3}} \le |f'(z)| \le \frac{1+|z|}{(1-|z|)^{3}},\\[12pt]
\dstyle \frac{|z|}{(1+|z|)^{2}} \le |f(z)| \le \frac{|z|}{(1-|z|)^{2}},\\[12pt]
\dstyle \frac{1-|z|}{1+|z|} \le \left|\frac{zf'(z)}{f(z)}\right| \le \frac{1+|z|}{1-|z|},
\end{array}
$$
for all $z \in \D$.
In each case, equality holds if and only if $f$ is a rotation of the Koebe function \eqref{Koebef}.
\end{thm}

\no
In particular, the third estimate implies that $f$ is locally uniformly bounded. 
Hence $\S$ forms a normal family.
Further, Hurwitz's theorem states that if a sequence of univalent functions on $\D$ converges to a holomorphic function locally uniformly on $\D$, then the limit function is univalent or constant. Since constant functions do not belong to $\S$ by the normalization, we conclude that $\S$ is compact in the topology of locally uniform convergence.

	%
			\subsection{Examples of quasiconformal extensions}
	%

For a given conformal mapping $f$ of a domain $D$, we say that $f$ has a \textbf{quasiconformal extension} to $\C$ if there exists a $k$-quasiconformal mapping $F$ such that its restriction $F|_{D}$ is equal to $f$.
For some fundamental conformal mappings, we can construct quasiconformal extensions explicitly.
Below we summarize such examples which are sometimes useful.
Some more examples can be found in \cite[p.78]{ImayoshiTaniguchi:1992}.
We remark that \eqref{qc_def} is written by the polar coordinates as
$$
\left|
\frac{ir\de_r f(re^{i\theta}) -\de_{\theta} f(re^{i\theta})}{ir\de_r f(re^{i\theta}) +\de_{\theta} f(re^{i\theta})}
\right| \le k,
$$
where $\de_r := \de /\de r$ and $\de_{\theta} := \de/\de \theta$.

 \begin{ex}
A very simple but important example is
$$
f(z) =\left\{
\begin{array}{ll}
\dstyle z + \frac{k}z, & |z| >1,\\[8pt]
z + k \bar{z}, & |z| \le 1.
\end{array} 
\right.
$$
where $k \in [0,1)$. Then $|f_{\bar{z}}/f_z| = k$ on $|z|\le 1$.
The case $k=1$ reflects the Joukowsky transform in $|z|>1$, though in this case $f$ is not a quasiconformal mapping any more.
\end{ex}

 \begin{ex} 
An identity mapping of $\D$ has trivially a quasiconformal extension. 
In fact, the following extension,
$$
f(z) =\left\{
\begin{array}{ll}
re^{i\theta}, & r < 1,\\
\phi(r) e^{i\theta}, & r \ge 1,
\end{array} 
\right.
$$
is given, where $\phi : [1, \infty) \to [1, \infty)$ is bi-Lipschitz continuous and injective with $\phi(1)=1$ and $\phi(\infty) =\infty$.
The maximal dilatation is given by
$$
|\mu_f| = \left|\frac{\phi(r)- r\phi'(r)}{\phi(r)+ r\phi'(r)}\right|.
$$
Let $M > 1$ be a Lipschitz constant. 
Then $1/M \le \phi'(r) \le M$ and $1/M \le \phi(r)/r \le M$ and therefore $1 \le r\phi'(r)/\phi(r) \le M^2$.
We conclude that the extension is $|\mu_f| \le |M^2 -1 |/|M^2 +1|$-quasiconformal.
\end{ex}

 \begin{ex}
Let $K(z) := (1+z)/(1-z)$ be the Cayley map and $P_{\beta}(z) := z^\beta$. 
For a fixed $\beta \in (0,2)$, the function 
$$
f(z) := (P_{\beta} \circ K)(z)
$$
maps $\D$ onto the sector domain $\Delta(-\beta, \beta) := \{z : -\pi\beta/2 < \arg z < \pi\beta/2\}$.
We shall construct a quasiconformal extension of $f$.
The function $g(z) := (-P_{2-\beta} \circ - K)(z)$ maps $\C\,\backslash\,\closure{\D}$ onto $\Delta(\beta, 4-\beta)$.
But in this case $f(e^{i\theta}) \neq g(e^{i\theta})$ for each $\theta \in (0,2\pi)$.
In order to sew these two functions on their boundaries, define $h(re^{i\theta}) := r^{\beta/(2-\beta)} e^{i\theta}$.
Then $(-P_{2-\beta} \circ h \circ - K)(z)$ takes the same value as $f$ on $\de\D$.
Hence it gives a quasiconformal extension of $f$.
A calculation shows that its maximal dilatation is $|1-\beta|$.
\end{ex}

 \begin{ex}
For a given $\lambda \in (-\pi/2, \pi/2)$, a function defined by
$$
f(re^{i\theta}) = e^{i\theta} \exp(e^{i\lambda} \log r)
$$
is a $\tan (\lambda / 2)$-quasiconformal mapping of $\C$ onto $\C$.
On the other hand, since the above $f$ maps a radial segment $[0,\infty)$ to a logarithmic spiral, it is not differentiable at the origin.
By calculation we have $|f|= \exp(\cos \lambda \log r)$ and $\arg f = \theta + \sin \lambda \log r$.
Therefore $f$ with a proper rotation gives a $\tan (\lambda / 2)$-quasiconformal extension for a function $f(z) = cz$ on $\D$ or $\D^* := \CC\,\backslash\,\closure{\D}$, where $c$ is some constant.
\end{ex}

 \begin{ex}
The functions $K(z) = z/(1-z)^{2}$ and $f(z) = z - z^{2}/2$ are typical examples in $\S$ which do not have any quasiconformal extensions
The first one is the Koebe function \eqref{Koebef} which maps $\D$ onto $\C\,\backslash\,(-\infty, -1/4]$.
There does not exist a homeomorphism which maps $\D^*$ onto $(-\infty, -1/4]$.
As for the second function, $\de\D$ is mapped to a cardioid which has a cusp at $z = 1$.

\end{ex}

	%
			\subsection{Extremal problems on $\S(k)$}
	%

In order to investigate the structure of the family of functions, the extremal problems sometimes provide us quite beneficial information.
The Bieberbach conjecture is one of the most known such problems. 
A similar problem for $\S(k)$ and $\Sigma(k)$, a subclass of $\Sigma$ such that all elements have $k$-quasiconformal extensions to $\CC$, were proposed, and many mathematicians have worked on this problem.
We note that in spite of such a circumstance, there are many open problems in this field including the coefficient problem.

Our argument is built on the following fact.

 \begin{thm}
$\S(k),\,\S^*(k)$ and $\Sigma(k)$ are compact families.
\end{thm}

K\"uhnau gave a fundamental contribution to the coefficient problem with the variational method.

 \begin{thm}[{\cite{Kuhnau:1969}}]
Let $f(z) = z + \sum_{n=2}^{\infty} a_nz^n \in \S(k)$ and $g(\zeta) = \zeta + \sum_{n=0}^{\infty}b_n z^{-n} \in \Sigma(k)$.
Then the followings hold; $|b_0| \le 2k$, $|b_1| \le k$ and $|a_3 - a_2^2| \le k$, in particular $|a_2| \le 2k$.
\end{thm}

\no
We note that in the case when $k=1$ we obtain estimates for the classes $\S$ and $\Sigma$.

As more general approach to this problem, the distortion theorem for bounded functional was studied.
We basically follow the description of the survey paper by Krushkal \cite[Chapter 3]{Krushkal:2005a}. 
The reader is also referred to \cite{KuKru:1983}.

Let $E \subset \CC$ be a measurable set whose complement $E^* := \CC \,\backslash\,E$ has positive measure, and set
$$
B^*(E) := \{ \mu  \in B(\CC) : \mu|_{E^*} =0\}.
$$
Denote by $Q(E)$ the family of normalized quasiconformal mappings $f_{\mu} : \CC \to \CC$ where $\mu \in B\*(E)$, and $Q_k(E) := \{f \in Q(E) : ||\mu_f|| \le k\}$ for a $k \in [0,1)$.
Now let $F : Q(E) \to \C$ be a non-trivial holomorphic functional, where holomorphic means that it is complex Gateaux differentiable.
Lastly, set $||F||_1 := \sup_{f \in Q(E)}|F(f)|$ and $||F||_k := \max_{f \in Q_k(E)}|F(f)|$.

 \begin{thm}
Let $F : (Q(E)) \to \C$ be bounded. Then we have
$
||F||_k \leq k||F||_1.
$
\end{thm}

\no
Some applications of the theorem are demonstrated in \cite[Chapter 3.4]{Krushkal:2005a}.
One of them is the distortion theorem for the class $S(k)$ (see also \cite[Corollary 7]{Gutl:1973});
$$
\dstyle \left(\frac{1-|z|}{1+|z|}\right)^k \le \left|\frac{zf'(z)}{f(z)}\right|\le  \left(\frac{1+|z|}{1-|z|}\right)^k.
$$
For more results and proofs, see \cite{Schober:1975}, \cite{Krushkal:2005a}, \cite{Krushkal:2005b}.

The estimate of $|a_2|$ for the class $\S^*(k)$ is obtained by Schiffer and Schober.

 \begin{thm}[{\cite{SchifferSchober:1976}}]
For all $f \in \S^*(k)$, we have the sharp estimate
$$
|a_2| \le 2-4\left(\frac{\arccos k}{\pi}\right)^{2}.
$$
For the sharp function, see \cite[Eq. (4.2)]{SchifferSchober:1976}
\end{thm}

\no
Since the class $\S^*(k)$ is closed with respect to the Koebe transform \eqref{Koebe_trans}, we have the fundamental estimate for $\S^*(k)$
$$
\left|\frac{zf''(z)}{f'(z)} -\frac{2|z|^2}{1-|z|^2}\right| \le \left(2-4\left(\frac{\arccos k}{\pi}\right)^{2}\right)\frac{2|z|}{1-|z|^2}.
$$
Following the standard argument for the class $\S$ (see Section 2.1, or \cite[pp.21-22]{Pom:1975}), we have distortions of $f$ and $f'$ for $\S^*(k)$.
We note that the same method as this is not valid for the class $\S(k)$ because the Koebe transform \eqref{Koebe_trans} does not fix $\infty$ except the case $\zeta =0$.

As is written before, while the coefficient problem has been completely solved in the class $\S$, the question remains open for the class $\S(k)$.
However, if we restrict ourselves to that $k$ is sufficiently small, then the sharp result is established by Krushkal.
 \begin{thm}[{\cite{Krushkal:1988, Krushkal:1995}}]
For a function $f(z) = z + a_2z^2 + \cdots \in \S(k)$, we have the sharp estimate
\begin{equation}
\label{coeff}
|a_n| \le \frac{2k}{n-1}
\end{equation}
for $k \le 1/(n^2+1)$.
\end{thm}

\no
The extremal function of the estimate \eqref{coeff} is given by
\begin{eqnarray*}
&&f_2(z) := \frac{z}{(1-kz)^2} \hspace{15pt} (k \in [0,1)),\\
&&f_n(z) := (f_2(z^{n-1}))^{1/(n-1)} = z + \frac{2k}{n-1}z^n + \cdots\hspace{15pt}n =3,4,\cdots.
\end{eqnarray*}

\no
To see $f_n \in \S(k)$, calculate $z f_n'(z)/f_n(z)$ and apply the quasiconformal extension criterion for starlike functions in Section 3.4.

	%
			\subsection{Sufficient conditions for $\S(k)$}
	%

Since Bers introduced a new model of the universal Teichm\"uller space, numerous sufficient conditions for the class $\S(k)$ have been obtained.
In this subsection we introduce only a few remarkable results.

In 1962, the first sufficient condition for $\S(k)$ was provided by Ahlfors and Weill.

 \begin{thm}[{\cite{AhlforsWeill:1962}}]
Let $f$ be a non-constant meromorphic function defined on $\D$ and $k \in [0,1)$ be a constant.
If $f$ satisfies $||S_f|| \le 2k$, then $f$ can be extended to a quasiconformal mapping $F$ to $\CC$.
In this case the dilatation $\mu_F$ is given by
$$
\mu_F(z) := \left\{
\begin{array}{lll}
\dstyle -\frac12 (|z|^2 -1)^2 S_F\left(\frac{1}{\bar{z}}\right)\frac{1}{\bar{z}^4},& |z|>1\\
0,&|z| < 1.
\end{array}
\right.
$$
\end{thm}

1972, Becker gave a sufficient condition in connection with the pre-Schwarzian derivative.
Later it was generalized by Ahlfors.
 \begin{thm}[{\cite{Ahlfors:1974}}]
\label{Ahlforscriterion}
Let $f \in \A$.
If there exists a $k \in [0,1)$ such that for a constant $c \in \C$ the inequality
\begin{equation}
\label{ahlfors}
\left|c|z|^2 + (1-|z|^2) \frac{f''(z)}{f'(z)}\right| \leq k
\end{equation}
holds for all $z \in \D$, then $f \in \S(k)$ .
\end{thm}

\no
The case when $c=0$ is due to Becker \cite{Becker:1972}. Remark that the condition $|c| \le k$ which was stated in the original form is embedded in the inequality \eqref{ahlfors} (see \cite{Hotta:2010b}).

It is known that many univalence criteria are refined to quasiconformal extension criteria.
For instance, Fait, Krzy\.{z} and Zygmunt proved the following theorem which is the refinement of the definition of strongly starlike functions (for the definition, see Section 3.3).
 \begin{thm}[{\cite{FaitKrzyzZ:1976}}]
Every strongly starlike functions of order $\alpha$ has a $\sin(\pi\alpha/2)$-quasiconformal extension to $\C$.
\end{thm}

\no
This is generalized to strongly spiral-like functions \cite{Sugawa:2012a}.
Some more results are obtained in \cite{Brown:1984, Hotta:2009} with explicit quasiconformal extensions which correspond to each subclass of $\S$.
In particular, in \cite{Hotta:2009} the research relies on the (classical) Loewner theory, which will be mentioned in the next section.

Sugawa approached this problem by means of the holomorphic motions with extended $\lambda$-Lemma (\cite{MSS:1983}, \cite{Slod:1991}).

 \begin{thm}[{\cite{Sugawa:1999a}}]
Let $k \in [0,1)$ be a constant.
For a given $f \in \A$, let $p$ denote one of the quantities $zf'(z)/f(z), 1 + zf''(z)/f'(z)$ and $f'(z)$.
If 
$$
\left|
\frac{1-p(z)}{1+p(z)}
\right| \le k
$$ 
for all $z \in \D$, then $f \in \S(k)$.
\end{thm} 

We note that in most of the sufficient conditions of quasiconformal extensions including the above theorems the case $k=1$ reflects univalence criteria.

\

	%
			\section{\bf Classical Loewner theory}
	%

The idea of the parametric representation method of conformal maps was introduced by L\"owner \cite{Loewner:1923}, and later developed by Kufarev \cite{Kufarev:1943} and Pommerenke \cite{Pom:1965}.
It describes a time-parametrized conformal map on $\D$ whose image is a continuously increasing simply connected domain.
The key point is that such a family can be represented by a partial differential equation.
Loewner's approach also made a significant contribution to quasiconformal extensions of univalent functions.
This method was discovered by Becker.

Since our focus in this note is on univalent functions with quasiconformal extensions, we will deal with Loewner chains in the sense of Pommerenke (see \cite{Pom:1975}). For one-slit maps as L\"owner originally considered, see e.g. \cite{delMonacoGumenyuk} which also contains a list of references.
For the classical theory, the reader is also referred to \cite[Chapter III-2]{Goluzin:1969}, \cite[Chapter IX-9]{Tsuji:1975}, \cite[Chapter 3]{Duren:1983}, \cite[Chapter 19]{Henrici:1986}, \cite[Chapter 7-8]{RosenblumRovnyak1994}, \cite[Chapter 7-8]{Hayman:1994}, \cite[Chapter 17]{Conway:1995}, \cite[Chapter 3]{GrahamKohr:2003}.

	%
			\subsection{Classical Loewner chains}
	%

Let $f_{t}(z) = e^t z + \sum_{n=2}^{\infty}a_{n}(t)z^{n}$ be a function defined on $\D \times [0,\infty)$.
$f_{t}$ is said to be a (\textbf{classical}) \textbf{Loewner chain} if $f_{t}$ satisfies the conditions (Fig. 2);

\begin{enumerate}
\def\labelenumi{\textit{\arabic{enumi}}.}
\item $f_{t}$ is holomorphic and univalent in $\D$ for each $t \in [0,\infty)$;
\item $f_{s}(\D) \subset f_{t}(\D)$ for all $0 \leq s < t < \infty$.
\end{enumerate}
\begin{figure}[h]
\begin{center}\label{figure01}
\includegraphics[width=270pt]{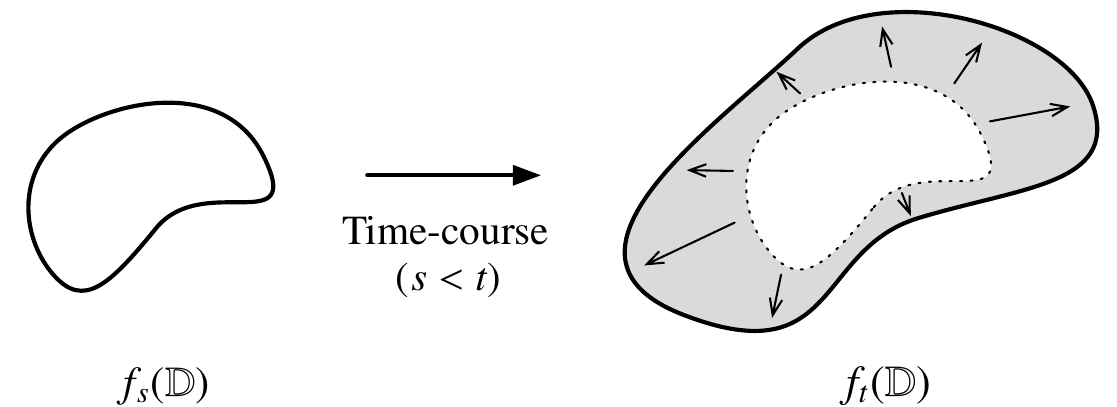}\\[5pt]
{Figure 2. The image of the unit disk $\D$ under $f_{t}$ expands continuously as $t$ increases.}
\end{center}
\end{figure}
One can also characterize it in the geometrical sense.
Let $\{D_t\}_{t \ge 0}$ be a family of simply connected domains having the following properties;
\begin{enumerate}
\def\labelenumi{\textit{\arabic{enumi}$'$}.}
\item $0 \in D_0$;
\item $D_s \subsetneq D_t$ for all $0 \le s < t < \infty$;
\item $D_{t_n} \to D_t$ if $t_n \to t < \infty$ and $D_{t_n} \to \C$ if $t_n \to \infty$ $(n \to \infty)$, in the sense of the kernel convergence.
\end{enumerate}
Then by the Riemann mapping theorem there exists a family of conformal mappings $\{f_t\}_{t \ge 0}$ such that $f_t(0) =0$ and $f_t'(0) > 0$ for all $t \ge 0$. 
We note that $f_{t}$ is continuous on $t \in [0,\infty)$, and $f_s'(0)< f_t'(0)$ for all $s <t $ (for otherwise by the Schwarz Lemma $f_t^{-1} \circ f_s$ is an identity, which contradicts $D_s \subsetneq D_t$).
So after rescaling as $f_0 \in \S$ and reparametrizing as $f_t'(0) = e^t$, we obtain a Loewner chain.

$f_{t}$ has a time derivative almost everywhere on $[0,\infty)$ for each fixed $z \in \D$. In fact, applying the distortion theorem for $\S$ (Theorem \ref{kdistortion}), the next estimate follows.
 \begin{lem}
\label{ftlemma}
For each fixed $z \in \D$, a Loewner chain $f_{t}$ satisfies
$$
|f_{t}(z) - f_{s}(z)| \leq \frac{8|z|}{(1-|z|)^{4}}|e^{t} - e^{s}|
$$
for all $0 \le s \le t < \infty$.
\end{lem} 

\no
Hence $f_{t}$ is absolutely continuous on $t \in [0,\infty)$ for all fixed $z \in \D$.

A necessary and sufficient condition for a Loewner chain is shown by Pommerenke.

\def\labelenumi{(\roman{enumi})}

 \begin{thm}[\cite{Pom:1965, Pom:1975}]\label{pom}
Let $0 < r_{0} \leq 1$.
Let $f_t(z) = e^{t}z + \sum_{n=2}^{\infty}a_{n}(t)z^{n}$ be a function defined on $\D \times [0,\infty)$.
Then $f_t$ is a Loewner chain if and only if the following two conditions are satisfied;
\begin{enumerate}
\item $f_t$ is holomorphic in $z \in \D_{r_{0}}$ for each $t \in [0,\infty)$, absolutely continuous in $t \in [0,\infty)$ for each $z \in \D_{r_{0}}$ and satisfies
\begin{equation}
\label{upper-bound-ft}
|f_t| \leq K_{0} e^{t} \hspace{20pt} (z \in \D_{r_{0}},\,t \in [0,\infty))
\end{equation}
for some positive constant $K_{0}$.
\item There exists a function $p(z,t)$ analytic in $z \in \D$ for each $t \in [0,\infty)$ and measurable in $t \in [0,\infty)$ for each $z \in \D$ satisfying
$$
\Re p(z,t) > 0 \hspace{20pt} (z \in \D,\,t \in [0,\infty))
$$
such that
\begin{equation} \label{LDE}
\dot{f}_t(z) =z f_t'(z) p(z,t)  \hspace{20pt} (z \in \D_{r_{0}},\,\textup{a.e.}~t \in [0,\infty))
\end{equation}
where $\dot{f} = \de f /\de t$ and $f' = \de f/\de z$.
\end{enumerate}
\end{thm}

The partial differential equation \eqref{LDE} is called the \textbf{Loewner-Kufarev PDE}, and the function $p$ in \eqref{LDE} is called a \textbf{Herglotz function}.

\begin{rem}
Inequality \eqref{upper-bound-ft} and the following classical result due to Dieudonn{\'e}  \cite{Dieudonne:1931} (for the proof, see e.g.  {\cite[p.259]{Tsuji:1975}}) ensure the existence of the uniform radius of univalence of $\{f_{t}\}_{t \ge 0}$;
\textit{Let $f$ be holomorphic on $\D$ satisfying $f(0)=0$, $f'(0) = a >0$ and $|f(z)| < M$ for all $z \in \D$.
Then $f$ is univalent on the disk $\{|z| < \rho <1\}$, where
$$
\rho := \frac{a}{M + \sqrt{M^{2} -a^{2}}}.
$$
}
Hence, although it is not written on the sufficient conditions of Theorem \ref{pom}, $f_{t}$ is implicitly assumed to be univalent on a certain disk whose radius is determined independently from $t \in [0,\infty)$.
\end{rem}

\begin{rem}
\eqref{LDE} describes an expanding flow of the image domain $f_{t}(\D)$ of a Loewner chain.
Indeed, \eqref{LDE} can be written as
$$
|\arg \dot{f}_t(z) - \arg z f_t'(z)| = |\arg p(z,t)| < \frac{\pi}{2}.
$$
It implies that the velocity vector $\dot{f}_t$ at a boundary point of the domain $f_{t}(\D_{r})$ points out of this set and therefore all points on $\de f_t(\D_r)$ moves to outside of $\closure{f_{t}(\D_{r})}$ when $t$ increases (Fig. 3).
\begin{figure}[h]
\begin{center}
\includegraphics[width=190pt]{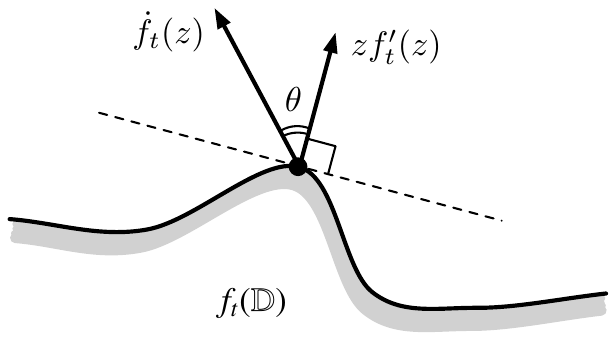}\\[5pt]
{Figure 3. The angle $\theta$ between the normal vector $z f_{t}'$ of the tangent line and the velocity vector $\dot{f}_{t}$ satisfies $|\theta| < \pi/2$.}
\end{center}
\end{figure}
\end{rem}

The next property is also important.
 \begin{thm}
For any $f \in \S$, there exists a Loewner chain $f_t$ such that $f_0 = f$.
\end{thm}

	%
			\subsection{Evolution families}
	%

In Loewner theory, a two-parameter family of holomorphic self-maps of the unit disk $(\varphi_{s,t}),\,0 \leq s \leq t < \infty$, called an \textbf{evolution family}, plays a key role.
To be precise, $(\varphi_{s,t})$ satisfies the followings;
\begin{enumerate}
\def\labelenumi{\textit{\arabic{enumi}}.}
\item $\varphi_{s,s}(z) = z$;
\item $\varphi_{s,t}(0) = 0$ and $\varphi_{s,t}'(0) = e^{s-t}$;
\item $\varphi_{s,t} = \varphi_{u,t}\circ \varphi_{s,u}$ for all $0 \leq s \leq u \leq t < \infty$.
\end{enumerate}

\no
We note that $\varphi_{s,t}$ is not assumed to be univalent on $\D$.
By means of the same idea as Lemma \ref{ftlemma}, we have the estimate for $0 \le s \le u \le t < \infty$,
$$
\begin{array}{ccc}
\dstyle |\varphi_{s,t}(z) - \varphi_{u,t}(z)| \le \frac{2|z|}{(1-|z|)^{2}}(1-e^{s-u}),\\[12pt]
\dstyle |\varphi_{s,u}(z) - \varphi_{s,t}(z)| \le 2|z|\frac{1+|z|}{1-|z|}(1-e^{u-t}),
\end{array}
$$
for all $z \in \D$.

For a Loewner chain $f_t$, the function $\varphi_{s,t} (z) := (f_t^{-1} \circ f_s)(z)$ defines an evolution family.
Since $f_t(\varphi_{s,t} (z)) = f_s$, differentiating both sides of the equation with respect to $t$ we have $\dot{f}_t(\varphi_{s,t})  + f _t'(\varphi_{s,t}) \dot{\varphi}_{s,t} = 0$.
Hence one can obtain by \eqref{LDE}
\begin{equation}
\label{LODE}
\dot{\varphi}_{s,t}(z) = -\varphi_{s,t}(z) p(\varphi_{s,t}(z),t).
\end{equation}
This is called the \textbf{Loewner-Kufarev ODE}.
The following is the basic result on existence and uniqueness of a solution of the ODE.

 \begin{thm}
\label{loewnerODE}
Suppose that a function $p(z,t)$ is holomorphic in $z \in \D$ and measurable in $t \in [0,\infty)$ satisfying $\Re p(z,t) >0$ for all $z \in \D$ and $t \in [0,\infty)$.
Then, for each fixed $z \in \D$ and $s \in [0, \infty)$, the initial value problem
$$
\frac{dw}{dt} = -wp(w,t)
$$
for almost all $t \in [s,\infty)$ has a unique absolutely continuous solution $w(t)$ with the initial condition $w(s) =z$.
If we write $\varphi_{s,t}(z) := w(t)$, then $\varphi_{s,t}$ is an evolution family and  univalent on $\D$. 
Further, the function $f_s(z)$ defined by
\begin{equation}
\label{limitation}
f_s(z) := \lim_{t \to \infty} e^t \varphi_{s,t}(z)
\end{equation}
exists locally uniformly in $z \in \D$ and is a Loewner chain.

Conversely, if $f_t$ is a Loewner chain and $\varphi_{s,t}$ is an evolution family associated with $f_t$ by $\varphi_{s,t} := f_t^{-1} \circ f_s$.
Then for almost all $t \in [s,\infty)$, $\varphi_{s,t}$ satisfies
$$
\frac{d \varphi_{s,t}}{dt} = -\varphi_{s,t} p(\varphi_{s,t}, t)
$$
for all $z \in \D$, and \eqref{limitation} is satisfied.
\end{thm}

In the first assertion of Theorem \ref{loewnerODE}, it may happen that two different Herglotz functions $p_1$ and $p_2$ generate the same evolution family $\varphi_{s,t}$.
Then $p_1(z,t) = p_2(z,t)$ for almost all $t \ge 0$.
Hence Theorem \ref{loewnerODE} says that there is a one-to-one correspondence between an evolution family and a Herglotz function in such a sense.

	%
			\subsection{Loewner chains and quasiconformal extensions}
	%

An interesting method connecting Loewner theory and quasiconformal extensions was obtained by Becker.

 \begin{thm}[\cite{Becker:1972}, \cite{Becker:1980}]\label{Beckerthm}
Suppose that $f_{t}$ is a Loewner chain for which $p(z,t)$ in \eqref{LDE} satisfying the condition 
\begin{equation}\label{U(k)}
p(z,t) \in U(k) :=
\dstyle
\left\{
w \in \C : \left|\frac{1-w}{1+w}\right| \leq k
\right\}
\end{equation}
i.e., $p(z,t)$ lies in the closed hyperbolic disk $U(k)$ in the right half-plane centered at 1 with radius $\textup{arctanh}\, k$, for all $z \in \D$ and almost all $t \ge 0$.
Then $f_t$ admits a continuous extension to $\closure{\D}$ for each $t \geq 0$ and the map $F$ defined by
\begin{equation}
\label{beckerequation}
F(re^{i\theta}) =
\left\{
\begin{array}{ll}
\dstyle f_0(re^{i\theta}), &\textit{if}\hspace{10pt}   r < 1, \\[5pt]
\dstyle f_{\log r}(e^{i\theta}), &\textit{if}\hspace{10pt}   r \geq 1,
\end{array} 
\right.
\end{equation}
is a $k$-quasiconformal extension of $f_{0}$ to $\C$.
\end{thm}

The idea of the theorem is the following.
By Koebe's 1/4-Theorem (Theorem \ref{1/4-theorem}), $f_t (\D)$ must contain the disk whose center is 0 with radius $e^t/4$.
Thus $f_t(\D)$ tends to $\C$ as $t \to \infty$.
This fact implies that the boundary $\de f_t(\D)$ runs throughout on $\C \,\backslash\,f_0(\D)$.
Therefore the mapping $F : \D^* \to \C \,\backslash\,f_0(\D)$ is constructed by \eqref{beckerequation} which gives a correspondence between the circle $\{|z| = e^t\}$ and the boundary $\de f_t(\D)$.
Its quasiconformality follows from the condition \eqref{U(k)}.

Betker generalized Theorem \ref{Beckerthm} by introducing an inverse version of Loewner chains.
Let $\omega_{t}(z) = \sum_{n=1}^{\infty}b_{n}(t)z^{n}$, $b_{1}(t) \neq 0$, be a function defined on $\D \times [0,\infty)$, where $b_{1}(t)$ is a complex-valued, locally absolutely continuous function on $[0,\infty)$.
Then $\omega_t$ is said to be an \textbf{inverse Loewner chain} if;

\begin{enumerate}
\def\labelenumi{\it \arabic{enumi}.}
\item $\omega_t$ is univalent in $\D$ for each $t \geq 0$;
\item $|b_1(t)|$ decreases strictly monotonically as $t$ increases, and $\lim_{t \to \infty}|b_1(t)| \to 0$;
\item $\omega_{s}(\D) \supset \omega_{t}(\D)$ for $0 \leq s < t < \infty$;
\item $\omega_0(z) = z$ and $\omega_s(0) = \omega_t(0)$ for $0 \leq s \leq t < \infty$.
\end{enumerate}

\no
$\omega_{t}$ also satisfies the partial differential equation
\begin{equation}\label{bb}
\dot{\omega}_t(z) = - z \omega_t'(z) q(z,t)\hspace{15pt}(z \in \D,\,\textup{a.e.}\, t \geq 0),
\end{equation}
where $q$ is a Herglotz function.
Conversely, we can construct an inverse Loewner chain by means of \eqref{bb} according to the following lemma:

 \begin{lem}
\label{BetkerLem}
Let $q(z,t)$ be a Herglotz function.
Suppose that $q(0,t)$ be locally integrable in $[0, \infty)$ with $\int_0^{\infty} \Re q(0, t) dt = \infty$. 
Then there exists an inverse Loewner chain $w_t$ with \eqref{bb}.
\end{lem} 

By applying the notion of an inverse Loewner chain, we obtain a generalization of Becker's result.

 \begin{thm}[{\cite{Betker:1992}}]
Let $k \in [0,1)$.
Let $f_t$ be a Loewner chain for which $p(z,t)$ in \eqref{LDE} satisfying the condition 
\begin{equation*}
\left|
\frac{p(z,t) -\closure{q(z,t)}}{p(z,t) + q(z,t)}
\right|
\leq k \hspace{20pt}
(z \in \D, \textup{a.e.}\, t \geq 0),
\end{equation*}
where $q(z,t)$ is a Herglotz function.
Let $\omega_t$ be the inverse Loewner chain which is generated with $q$ by Lemma \ref{BetkerLem}. 
Then $f_t$ and $\omega_t$ are continuous and injective on $\closure{\D}$ for each $t \geq 0$, and $f_0$ has a $k$-quasiconformal extension $F : \C \to \C$ which is defined by
\begin{equation*}
F\left(\frac{1}{\closure{\omega_t(e^{i\theta})}}\right) = f_t(e^{i\theta}) \hspace{20pt} (\theta \in [0,2\pi), t \geq 0).
\end{equation*}
\end{thm}

\no
We obtain Becker's result for $q(z,t) = 1$. 
In this case an inverse Loewner chain is given by $\omega_t(z) = e^{-t}z$.
Further, choosing $\omega$ as $p = q$, we have the following corollary:

 \begin{cor}[{\cite{Betker:1992}}]
\label{betkercor}
Let $\alpha \in [0,1)$.
Suppose that $f_{t}$ is a Loewner chain for which $p(z,t)$ in \eqref{LDE} satisfies
\begin{equation*}
p(z,t) \in \Delta(-\alpha, \alpha) =\left\{ z : -\frac{\alpha \pi}2 \leq \arg z \leq \frac{\alpha \pi}2\right\}
\end{equation*}
for all $z \in \D$ and almost all $t \in [0,\infty)$.
Then $f_t$ admits a continuous extension to $\closure{\D}$ for each $t \geq 0$ and $f_0$ has a $\sin \alpha\pi /2$-quasiconformal extension to $\C$.
\end{cor}

Corollary \ref{betkercor} does not include Theorem \ref{Beckerthm} in view of the dilatation of the extended quasiconformal map.
In fact, the following relation holds;
$$
U(k) \subset \Delta(-k_0, k_0) \hspace{10pt}\textup{where} \hspace{10pt} k_0 := \frac{2}{\pi} \arcsin\left(\frac{2k}{1+k^2}\right) \ge k.
$$
\no
Remark that $k_0=k$ if and only if $k=0$.

In contrast to Becker's quasiconformal extension theorem, the theorem due to Betker does not always provide an explicit quasiconformal extension.
The reason is based on the fact that it is difficult to express an inverse Loewner chain $\omega_t$ which has the same Herglotz function as a given Loewner chain $f_t$ in an explicit form.
For details, see \cite[Section 5]{HottaWang:2014}

	%
\subsection{Applications to the theory of univalent functions}
	%

Here we will see some applications of Theorem \ref{pom} and Theorem \ref{Beckerthm}.
In order to find out explicit Loewner chains which corresponds to the typical subclasses of $\S$, we need to observe their geometric features.
Some Loewner chains are not normalized as $f'(0) = e^t$.
In \cite{Hotta:2010a}, it is discussed that Theorem \ref{pom} and Theorem \ref{Beckerthm} work well without such a normalization.
In fact, a Loewner chain is generalized for a function $f_t(z) = \sum_{n=1}^{\infty} a_t(z) z^t$ where $a_1(t) \neq 0$ is a complex-valued, locally absolutely continuous function on $t \in [0,\infty)$ with $\lim_{n \to \infty} |a_1(t)| = \infty$. Further, either the condition that $|a_1(t)|$ is strictly increasing with respect to $t \in [0,\infty)$, or $f_s(\D) \subsetneq f_t(\D)$ for all $0 \le s < t < \infty$ should be assumed.

\subsubsection{Convex functions}
A function $f \in \S$ is said to be \textbf{convex} and belongs to $\K$ if $f(\D)$ is a convex domain. 
It is known that $f \in \K$ if and only if 
$$\Re \left[1 + \frac{zf''(z)}{f'(z)}\right]>0$$ for all $z \in \D$.
A flow of the expansion for a convex function is considered as following. If a boundary point $\zeta \in \de f(\D)$ moves to the direction of their normal vector $\zeta f'(\zeta)$ according to the parameter $t$ increases, then $\zeta$ always runs on the complement of $f(\D)$ and their trajectories do not cross each other.
In view of this, it is natural to set a Loewner chain as
\begin{equation}
\label{LCconvex}
f_t(z) = f(z) + t \cdot zf'(z)
\end{equation}
Then we have $1/p(z,t) = 1 + t \cdot [1 + (zf''(z)/f'(z))]$ and hence $f_t$ is a Loewner chain if $f \in \K$.

\subsubsection{Starlike functions}
Next, consider a \textbf{starlike function} (with respect to 0), i.e., a function $f \in \S$ such that for every $z \in \D$ the segment connecting $f(z)$ and 0 lies in $f(\D)$. Denote by $\S^*$ the family of starlike functions.
An analytic characterization for starlike functions is $$\Re\left[ \frac{zf'(z)}{f(z)}\right] >0$$ for all $z \in \D$.
It follows from the definition that for a boundary point $\zeta \in\de f(\D)$, the ray $\{ t \zeta : t  \ge 1\}$ always lies in the exterior of $f(\D)$. 
Hence the possible chain for $\S^*$ is
\begin{equation}
\label{LCstarlike}
f_t (z) := e^t f(z).
\end{equation}
A simple calculation shows that $1/p(z,t) = zf'(z)/f(z)$ and therefore $f_t$ is a Loewner chain if $f \in \S^*$.
In the case of \textbf{spiral-like functions}, i.e., functions $f \in \S$ defined by the condition $$\Re \left[e^{-i\lambda} \frac{zf'(z)}{f(z)}\right] > 0$$ for some $\lambda \in (-\pi/2, \pi/2)$, a Loewner chain is given by
\begin{equation}
\label{LCspiral}
f_t(z) := e^{ct} f(z)
\end{equation}
with $c := e^{i\lambda}$ whose trajectories draw logarithmic spirals.
The case $\lambda=0$ corresponds to starlike functions.

\subsubsection{Close-to-convex functions}
For a given $f \in \S$, if there exists a $g \in \S^*$ such that $$\Re \left[ e^{-i\lambda}\frac{zf'(z)}{g(z)} \right]>0$$ for some $\lambda \in (-\pi/2,\pi/2)$ and all $z \in \D$, then $f$ is said to be \textbf{close-to-convex} and we denote by $f \in \CTC$.
The image $f(\D)$ by a close-to-convex function is known to be a \textbf{linearly accessible domain}, namely, $\C\,\backslash\,f(\D)$ is the union of closed half-lines which are mutually disjoint except their end points.
$f$ is said to be \textbf{linearly accessible} if $f(\D)$ is a linearly accessible domain.

A Loewner chain corresponding to the class $\CTC$ is given by
\begin{equation}
\label{ctcchain}
f_t(z) := f(z) + t \cdot e^{i\lambda} g(z).
\end{equation}
Then $1/p(z,t) = e^{-i\lambda}(zf'(z)/g(z)) + t (zg'(z)/g(z))$ and hence $\Re p(z,t) >0$ for all $z \in \D$ and $t \ge 0$.
The validity of the chain \eqref{ctcchain} is given by the following consideration.

Below we consider the case $\lambda =0$.
Take a fixed $\rho \in (0,1)$ and set $f_{\rho}(z) := f(\rho z)/\rho$ and $g_\rho(z) := g(\rho z)/\rho$.
Then $f_t^{\rho} := f_{\rho} + t g_{\rho}$ is well-defined on $\closure{\D}$.
For each boundary point $\zeta_0 \in \de\D$, 
$$
\gamma_{\zeta_0} := \{f_t^{\rho}(\zeta_0) : t \in [0,\infty)\}
$$ 
defines a half-line with an inclination of $\arg g_{\rho}(\zeta_0)$.
Let $\zeta_1 \in \de\D$ be another boundary point with $\zeta_1 \neq \zeta_0$.
Since $f_t^{\rho}$ is a Loewner chain, $\gamma_{\zeta_0}$ and $\gamma_{\zeta_1}$ do not have any intersection.
Further, by the property $f_t^{\rho}(\D) \to \C$ as $t \to \infty$, $\gamma_{\zeta}$ runs throughout $\C\,\backslash\,f_{\rho}(\D)$ if $\arg \zeta$ is taken from 0 to $2\pi$.
Therefore $\bigcup_{\zeta \in \de\D} \gamma_{\zeta} = \C\,\backslash\,f_{\rho}(\D)$ which proves that every $f_{\rho} \in \CTC$ is linearly accessible.
It is known that the family of linearly accessible functions $f \in \S$ is compact in the topology of locally uniform convergence (\cite{Biernacki:1936}).
Hence we conclude that $f = \lim_{\rho \to 1}f_{0}^{\rho}\in \CTC$ is linearly accessible.

One can prove it without compactness of the family of linearly accessible functions.
Let $p_{\zeta}[f]$ be the prime end defined on a domain $f(\D)$ corresponding to a boundary point $\zeta \in \de\D$ and $I_{\zeta}[f]$ be the impression of the prime end $p_{\zeta}[f]$.
It is known that there is a one-to-one correspondence among $\zeta$, $p_{\zeta}[f]$ and $I_{\zeta}[f]$ (see \cite[Chapter 2]{Pom:1992boundary}).
Since $g$ is starlike, for all $w_g \in I_{\zeta_0}[g]\,\backslash\{\infty\}\,$, $\arg w_g$ reflects one real value.
Then redefine $\gamma_{\zeta_0}$ as a family of rays (may consist of only one ray) by
$$
\gamma_{\zeta_0} := \{w_{f} + t \exp(i \arg  w_g) : w_f \in I_{\zeta_0}[f]\,\backslash\{\infty\}\,,\,w_g \in I_{\zeta_0}[g]\,\backslash\{\infty\}\,,\, t \in [0,\infty)\}.
$$
Then $\bigcup_{\zeta \in \de\D}\gamma_{\zeta} = \C \,\backslash\,f(\D)$, for otherwise there exists a point $z \in \C \,\backslash\,f(\D)$ such that $z \not\in \gamma_{\zeta}$ for any $\zeta \in \de\D$ which contradicts the fact that $f_t$ is a Loewner chain.
By choosing proper components of $\bigcup_{\zeta \in \de\D}\gamma_{\zeta}$, a union of closed half-lines for that $f(\D)$ is a linearly accessible domain is given.

The Noshiro-Warschawski class is known as the special case of close-to-convex functions.
Noshiro \cite{Noshiro:1934} and Warschawski \cite{Warschawski:1935} independently proved that if a function $f \in \A$ satisfies $$\Re f'(z)>0$$ for all $z \in \D$, then $f \in \S$ (see e.g. \cite{HottaWang:2014}). 
We denote the family of such functions by $\mathcal{R}$.
Choosing $g(z) = z$ and $ \lambda =0$ in \eqref{ctcchain}, we have the chain
\begin{equation}
\label{LCnw}
f_t(z) := f(z) + t z
\end{equation}
which proves $\mathcal{R} \subset  \CTC \subset \S$.
By this consideration, the following property is derived.
 
\begin{prop}
For a function $f \in \mathcal{R}$, if the boundary of $f(\D)$ is locally connected, then $e^{i\theta} \mapsto f(e^{i\theta}) \in \C$ is one-to-one.
\end{prop}

Further, we can make use of \eqref{LCnw} to observe the shape of $f(\D)$ for an $f \in \mathcal{R}$.
We assume that the boundary of $f(\D)$ is locally connected.
Then the half-line $\gamma_{e^{i\theta}} := \{f(e^{i\theta}) + t e^{i\theta} : t \in [0,\infty)\}$ is well-defined.
Since the inclination of $\gamma_{e^{i\theta}}$ is exactly $\theta$, we obtain the following property for $\mathcal{R}$;
 
\begin{prop}
Let $f \in \S$. If $f(\D)$ contains some sector domain in $\C$, then $f$ does not belong to $\mathcal{R}$.
\end{prop}

\no
For example, $f(z) = ((1+z)/(1-z) -1)/2$ maps $\D$ onto the half-plane.
Hence we immediately conclude that $f \notin \mathcal{R}$ (of course in this case it is easy to see that $f$ does not satisfy $\Re f' >0$ by calculation).

\subsubsection{Bazilevi\v c functions}

For real constants $\alpha >0$ and $\beta \in \R$, set $\gamma = \alpha + i \beta$.
In 1955, Bazilevi\v c \cite{Bazilevic:1955e} showed that the function defined by
$$
f(z)=
\left[
(\alpha + i \beta)
\int_{0}^{z}
h(u) g(u)^{\alpha} u^{i\beta-1}du
\right]^{1/(\alpha + i \beta)}
$$
where $g$ is a starlike univalent function and $h$ is an analytic function with $h(0)=1$ satisfying $\Re (e^{i\lambda}h) > 0$ in $\D$ for some $\lambda \in \R$ belongs to the class $\S$.
It is called a \textbf{Bazilevi\v c function of type $(\alpha, \beta)$} and we denote by $\B(\alpha, \beta)$ the family of Bazilevi\v c functions of type $(\alpha, \beta)$.
A simple observation shows that $f \in \B(\alpha, \beta)$ if and only if
$$
\Re \left\{e^{i \lambda}\frac{zf'(z)}{f(z)}\left(\frac{f(z)}{g(z)}\right)^{\alpha}\left(\frac{f(z)}{z}\right)^{i\beta} \right\} > 0 \hspace{20pt}(z \in \D)
$$
for some $g \in \S^*$.
A Loewner chain for the class $\B(\alpha, \beta)$ is known (\cite[p.166]{Pom:1965}) as
\begin{equation}
\label{LCBazilevic}
f_t(z) = 
	\left(
	f(z)^{\gamma} + t \cdot \gamma  g(z)^{\alpha} z^{i \beta}
	\right)^{1/\gamma}.
\end{equation}

By using the previous argument for close-to-convex functions, we can derive some geometric features for the class $\B(\alpha, \beta)$.
We consider the simple case that the boundaries of $f(\D)$ and $g(\D)$ are locally connected.
Then for each point $\zeta_0 \in \de\D$, the curve 
$
\{ 
	\delta_{\zeta_0}(t) := (f(\zeta_0)^{\gamma} + t \cdot \gamma  g(\zeta_0)^{\alpha} {\zeta_0}^{i \beta})^{1/\gamma} : t \in [0,\infty) 
\}
$ 
is defined. 
Hence $f(\D)$ is described as the complement of a union of such curves.

Observe the behavior of the curve. 
If $\beta >0$ (or $\beta <0$), then it draws an asymptotically similar curve as a logarithmic spiral which evolves counterclockwise (or clockwise).
On the other hand, in the case when $\beta =0$, firstly it draws a spiral, then tends to a straight line as $t$ gets large.
In both cases, the curvature $d_t \arg \delta_{\zeta_0}'(t) = \Im [\delta_{\zeta_0}''(t)/\delta_{\zeta_0}'(t)]$ is always positive or negative.
From this fact one can construct functions which do not belong to any $\B(\alpha, \beta)$ easily.
Consider a slit domain $\C\,\backslash\,\gamma$.
If the curvature of the slit $\gamma$ takes both positive and negative values (ex. $\gamma = \{x+ iy : y=\sin x\; \text{ and }\; x > 0\}$), or $\gamma$ is not smooth (ex. $\gamma = \{x \ge 0\} \cup \{iy : y \in (0,1)\}$), then such slit domains cannot be images of $\D$ under any $f \in \B(\alpha, \beta)$.

	%
\subsection{Applications to quasiconformal extensions}
	%

Applying Theorem \ref{Beckerthm} to the chains \eqref{LCconvex}, \eqref{LCstarlike}, \eqref{LCspiral}, \eqref{ctcchain}, \eqref{LCnw} and \eqref{LCBazilevic} we obtain quasiconformal extension criteria for each subclass of $\S$ with explicit extensions. 
In this case the chains \eqref{LCconvex}, \eqref{ctcchain}, \eqref{LCnw} and \eqref{LCBazilevic} should be reparametrized by $e^t-1$.
The theorems can be found in \cite{Hotta:2009, Hotta:2010a, HottaWang, Hotta:phi}.
Further, by Theorem \ref{betkercor} with the chains \eqref{LCstarlike} and \eqref{LCspiral} we obtain quasiconformal extension criteria given by \cite{FaitKrzyzZ:1976} and \cite{Sugawa:2012a}.
For an explicit extension of these cases, see \cite{HottaWang:2014}.

The other typical example is Theorem \ref{Ahlforscriterion}, Ahlfors's quasiconformal extension criterion.
It can be obtained by Theorem \ref{Beckerthm} with the chain
$$
f_t(z) := f(e^{-t}z) + \frac1{1+c}(e^t-e^{-t})zf'(e^{-t}z),
$$
for then
$$
\frac{1-p(z,t)}{1+p(z,t)}
\,\,=\,\,
\frac{zf_t'(z) - \dot{f}_t(z)}{zf_t'(z) + \dot{f}_t(z)}
\,\,=\,\,
c \frac1{e^{2t}} + \left(1 - \frac1{e^{2t}}\right) \frac{e^{-t}z f'(e^{-t}z)}{f''(e^{-t}z)}.
$$

\

	%
\section{\bf Modern Loewner theory}
	%

Recently a new approach to treat evolution families and Loewner chains in a general framework has been suggested by Bracci, Contreras, D\'iaz-Madrigal and Gumenyuk (\cite{BracciCD:evolutionI}, \cite{BracciCD:evolutionII}, \cite{MR2789373}).
It enables us to describe a variety of the dynamics of one-parameter family of conformal mappings.
In this section we outline the theory of generalized evolution families and Loewner chains.
The key fact is that there is an (essentially) one-to-one correspondence among evolution families and Herglotz vector fields.
We also present some results about generalized Loewner chains with quasiconformal extensions.

\subsection{Semigroups of holomorphic mappings}

Let $D$ be a simply connected domain in the complex plane $\C$.
We denote the family of all holomorphic functions on $D$ by $\hol(D,\C)$.
If $f \in \hol(\D, \C)$ is a self-mapping of $\D$, then we will denote the family of such functions by $\hol(\D)$.

An easy consequence of the well-known Schwarz-Pick Lemma, $f \in \hol(\D) \backslash \{ \textup{id}\}$ may have at most one fixed point in $\D$.
If such a point exists, then it is called the \textbf{Denjoy-Wolff point} of $f$.
On the other hand, if $f$ does not have a fixed point in $\D$, then the Denjoy-Wolff theorem (see e.g. \cite{ElinShoikhet:2010}) claims that there exists a unique boundary fixed point $\angle \lim_{z \to \tau}f(z) = \tau \in \de\D$ such that the sequence of iterates $\{f^n\}_{n \in \N}$ converges to $\tau$ locally uniformly, where $\angle \lim$ denotes an angular (or non-tangential) limit, and $f^n$ an $n$-th iterate of $f$, namely, $f^1 := f$ and $f^n := f \circ f^{n-1}$.
In this case the boundary point $\tau$ is also called the \textbf{Denjoy-Wolff point}.
Remark that a boundary fixed point is not always the Denjoy-Wolff point.
A simple example is observed with a holomorphic automorphism of $\D$, $f(z) = (z+a)/(1+\bar{a}z)$ with $a \in \D\backslash \{0\}$.
$f$ has two boundary fixed points $\pm a/|a|$, but only one $a/|a|$ can be the Denjoy-Wolff point.

A family $(\phi_t)_{t \geq 0}$ of holomorphic self-mappings of $\D$ is called a \textbf{one-parameter semigroup} if;

\begin{enumerate}
\def\labelenumi{\it{\arabic{enumi}}.}
\item $\phi_0  = id_{\D}$;
\item $\phi_t \circ \phi_s = \phi_{s+t}$ for all $s, t \in [0, \infty)$;
\item $\lim_{t \to 0^+} \phi_t (z) = z$ locally uniformly on $\D$;
\end{enumerate}

\no
In the definition, only right continuity at 0 is required.

The following theorem is fundamental in the theory of one-parameter semigroups.

\begin{thm}
Let $(\phi_t)_{t \ge 0}$ be a one-parameter semigroup of holomorphic self-mappings of $\D$.
Then for each $z \in \D$ there exists the limit
\begin{equation}
\label{fund_semi01}
\lim_{t \to 0^+}\frac{\phi_t(z) -z}{t} =: G(z)
\end{equation}
such that $G \in \hol(\D, \C)$. The convergence in \eqref{fund_semi01} is uniform on each compact subset of $\D$.
Moreover, the semigroup $(\phi_t)_{t \ge 0}$ can be defined as a unique solution of the Cauchy problem 
\begin{equation*}
\dstyle \frac{d\phi_t(z)}{dt} = G(\phi_t(z)) \hspace{15pt} (t \ge 0) 
\end{equation*}
with the initial condition $\phi_0(z) = z$.
\end{thm}

\no
The above function $G \in \hol(\D, \C)$ is called the \textbf{infinitesimal generator} of the semigroup.
Various criteria which guarantee that a homeomorphic function $G \in \hol(\D, \C)$ is the infinitesimal generator are known.
As one of them, in 1978 Berkson and Porta gave the following fundamental characterization.

\begin{thm}[{\cite{BerksonPorta:1978}}]
A holomorphic function $G \in \hol(\D, \C)$ is an infinitesimal generator if and only if there exists a $\tau \in \closure{\D}$ and a function $p \in \hol(\D, \C)$ with $\Re p(z) \geq 0$ for all $z \in \D$ such that
\begin{equation}\label{BPformula}
G(z) = (\tau-z)(1-\bar{\tau} z)p(z)
\end{equation}
for all $z \in \D$. 
\end{thm}

The equation \eqref{BPformula} is called the \textbf{Berkson-Porta representation}.
In fact, the point $\tau$ in \eqref{BPformula} is the Denjoy-Wolff point of the one-parameter semigroup generated with $G$.

	%
	%
\subsection{Generalized evolution families in the unit disk}
	%

We have discussed in Section 3.1 that a Loewner chain $f_t$ (in the classical sense) defines a function $\varphi_{s,t} := f_t^{-1} \circ f_s : \D \to \D$ which is called an evolution family. 
Recently, this notion and one-parameter semigroups are unified and generalized as following.

 \begin{definition}[{\cite[Definition 3.1]{BracciCD:evolutionI}}]
			\label{EVd}
A family of holomorphic self-maps of the unit disk $(\varphi_{s,t}),\,0 \leq s \leq t < \infty$, is an \textbf{evolution family} if;
\begin{enumerate}
\def\labelenumi{EF\arabic{enumi}.}
\item $\varphi_{s,s}(z) = z$;
\item $\varphi_{s,t} = \varphi_{u,t}\circ \varphi_{s,u}$ for all $0 \leq s \leq u \leq t < \infty$;
\item for all $z \in \D$ and for all $T>0$ there exists a non-negative locally integrable function $k_{z, T} :[0,T] \to  \R_{\ge 0}$ such that
$$
|\varphi_{s,u}(z) - \varphi_{s,t}(z)| \leq \int_{u}^{t} k_{z, T}(\xi) d\xi
$$
for all $0 \leq s \leq u \leq t \leq T$.
\end{enumerate}
\end{definition}

\no
We denote the family of evolution families by $\EF$.
%
%
%

\begin{rem}
If $(\phi_{t}) \subset  \hol(\D)$ is a one-parameter semigroup, then $(\varphi_{s,t})_{0 \le s \le t < \infty} := (\phi_{t-s})_{0 \le s \le t < \infty}$ forms an evolution family.
\end{rem}

\begin{rem}
In \cite{BracciCD:evolutionI} and \cite{MR2789373}, the definitions of evolution families and some other relevant notions contain an integrability order $d\in[1,+\infty]$. 
Since this parameter is not important for the discussions in this article, we assume that $d=1$ which is the most general case of the order.
\end{rem}

Some fundamental properties of $\EF$ are derived as follows.

\begin{thm}[{\cite[Proposition 3.7, Corollary 6.3]{BracciCD:evolutionI}}]
Let $(\varphi_{s,t}) \in \EF$. 
\begin{enumerate}
\item $\varphi_{s,t}$ is univalent in $\D$ for all $0 \leq s \leq t < \infty$.
\item For each $z_0 \in \D$ and $s_0 \in [0,\infty)$, $\varphi_{s_0,t}(z_0)$ is locally absolutely continuous on $t \in [s_0,\infty)$.
\item For each $z_0 \in \D$ and $t_0 \in (0,\infty)$, $\varphi_{s,t_0}(z_0)$ is absolutely continuous on $s \in [0,t_0]$.
\end{enumerate}
\end{thm}

Next, we extend the notion of infinitesimal generators to the same structure as evolution families.

\begin{definition}[{\cite[Definition 4.1, Definition 4.3]{BracciCD:evolutionI}}]
\textbf{A Herglotz vector field} on the unit disk $\D$ is a function $G : \D \times [0,\infty) \to \C$ with the following properties;
\begin{enumerate}
\def\labelenumi{HV\arabic{enumi}.}
\item for all $z \in \D$, $G(z,\,\cdot\,)$ is measurable on $[0,\infty)$;
\item for any compact set $K \subset \D$ and for all $T>0$, there exists a non-negative locally integrable function $k_{K,T} :[0,T] \to \R_{\ge 0}$ such that
$$
|G(z,t)| \leq k_{K,T}(t)
$$
for all $z \in K$ and for almost every $t \in [0,T]$;
\item for almost all $t \in [0,\infty)$, $G(\,\cdot\,t)$ is an infinitesimal generator.
\end{enumerate}
\end{definition}
 
\no
We denote  by $\HV$ the family of all Herglotz vector fields.

The following theorem states the relation between $(\varphi_{s,t}) \in \EF$ and $G \in \HV$.
In what follows, an \textbf{essentially unique} $f(x)$ means if there exists another function $g(x)$ which satisfies the statement then $f(x) = g(x)$ for almost all $x$.

\begin{thm}[{\cite[Theorem 5.2, Theorem 6.2]{BracciCD:evolutionI}}]
\label{EFtoHVthm}
For any $(\varphi_{s,t}) \in \EF$, there exists an essentially unique $G \in \HV$ such that
\begin{equation}
\label{EFtoHV}
\frac{d\varphi_{s,t}(z)}{d t} = G(\varphi_{s,t}(z), t)
\end{equation}
for all $z \in \D$, all $s \in [0,\infty)$ and almost all $t \in [s,\infty)$.
Conversely, for any $G \in \HV$, a family of unique solutions of \eqref{EFtoHV} with the initial condition $\varphi_{s,s}(z) =z$ generates an evolution family.
\end{thm}

The similar mutual characterization holds between a Herglotz vector field and a pair of the generalized Denjoy-Wolff point $\tau$ and a generalized Herglotz function.

 \begin{definition}[{\cite[Definition 4.5]{BracciCD:evolutionI}}]
	
			\label{HFd}
			
\textbf{A Herglotz function} on the unit disk $\D$ is a function $p : \D \times [0,\infty) \to \C$ with the following properties;
\begin{enumerate}
\def\labelenumi{HF\arabic{enumi}.}
\item for all $z \in \D$, $p(z,\,\cdot\,)$ is locally integrable on $[0,\infty)$;
\item for almost all $t \in [0,\infty)$, $p(\,\cdot\,,t)$ is holomorphic on $\D$;
\item $\Re p(z,t) \geq 0$ for all $z \in \D$ and almost all $t \in [0,\infty)$. 
\end{enumerate}
\end{definition}

\no
We denote $\HF$ the family of all Herglotz functions.

\begin{thm}[{\cite[Theorem 4.8]{BracciCD:evolutionI}}]
\label{HVtoBPthm}
Let $G \in \HV$.
Then there exists an essentially unique measurable function $\tau : [0,\infty) \to \closure{\D}$ and $p \in \HF$ such that
\begin{equation}\label{HVtoBP}
G(z,t) = (\tau(t)-z) (1-\closure{\tau(t)} z) p(z,t)
\end{equation}
for all $z \in \D$ and almost all $t \in [0,\infty)$.
Conversely, for a given measurable function $\tau : [0,\infty) \to \closure{\D}$ and $p \in \HF$, the equation \eqref{HVtoBP} forms a Herglotz vector field.
\end{thm}

\no
For convenience, we call the above measurable function $\tau : [0,\infty) \to \overline{\D}$ the \textbf{Denjoy-Wolff function} and denote by $\tau \in \DW$.
A pair $(p, \tau)$ of $p \in \HV$ and $\tau \in \DW$ is called the \textbf{Berkson-Porta data} for $G \in \HV$.
We denote the set of all Berkson-Porta data by $\BP$.
Hence, there is an essentially one-to-one correspondence among $(\varphi_{s,t}) \in \EF$, $G \in \HV$ and $(p,\tau) \in \BP$. 
In particular, the relation of $\varphi_{s,t}$ and $(p,\tau)$ is described by the ordinary differential equation
\begin{equation}\label{evolution}
\dot{\varphi}_{s,t}(z) = (\tau(t) - \varphi_{s,t}(z))(1 - \closure{\tau(t)}\varphi_{s,t}(z)) p(\varphi_{s,t}(z),t)
\end{equation}
which incorporates the Loewner-Kufarev ODE \eqref{LODE} and the Berkson-Porta representation \eqref{BPformula} as special cases.

	%
			\subsection{Generalized Loewner chains}
	%

According to the notion of evolution families, Loewner chains are also generalized as follows.

 \begin{definition}[{\cite[Definition 1.2]{MR2789373}}]
	
			\label{LdLC}
			
A family of holomorphic functions $(f_{t})_{t \geq 0}$ on the unit disk $\D$ is called a \textbf{Loewner chain} if;
\def\labelenumi{LC\arabic{enumi}.}
\begin{enumerate}
	\item $f_{t} : \D \to \C$ is univalent for each $t \in [0,\infty)$;
	\item $f_{s}(\D) \subset f_{t}(\D)$ for all $0 \leq s < t < \infty$;
	\item for any compact set $K \subset \D$ and all $T>0$, there exists a non-negative function $k_{K,T} : [0,T] \to \R_{\ge 0}$ such that
	\begin{equation*}
		|f_{s}(z) - f_{t}(z)| \leq \int_{s}^{t} k_{K,T}(\xi) d\xi
	\end{equation*}
	for all $z \in K$ and all $0 \leq s \leq t \leq T$.
\end{enumerate}
Further, a Loewner chain will be said to be \textbf{normalized} if $f_{0} \in \S$.
\end{definition}

We denote a family of Loewner chains by $\LC$.
Remark that in Definition \ref{LdLC}, any assumption is not required to $f_t(0)$ and $f_{t}'(0)$.
It implies that a subordination property that $f_s(\D_r) \subset f_t(\D_r)$ for all $r \in (0,1)$ and $0 \le s<t < \infty$ does not hold any longer in general.
Further, we even do not know whether the \textbf{Loewner range} $$\Omega[(f_t)] := \bigcup_{t \geq 0} f_t(\D)$$ is the whole complex plane or not.

The next theorem gives a relation between Loewner chains and evolution families.

 \begin{thm}[{\cite[Theorem 1.3]{MR2789373}}]
	
			\label{LEtoEV}
			
For any $(f_t) \in \LC$, if we define
\begin{equation*}\label{LEtoEV1}
	\varphi_{s,t}(z) := (f_t^{-1} \circ f_s) (z) \hspace{15pt}(z \in \D,\, 0 \le s \le t < \infty)
\end{equation*}
then $(\varphi_{s,t}) \in \EF$. Conversely, for any $(\varphi_{s,t}) \in \EF$, there exists an $(f_t) \in \LC$ such that the following equality holds
\begin{equation}\label{LEtoEV2}
(f_t \circ \varphi_{s,t})(z) = f_s(z) \hspace{15pt}(z \in \D,\,0 \le s \le t < \infty).
\end{equation}
\end{thm}

\no
Differentiate both sides of \eqref{LEtoEV2} with respect to $t$ then $f_t'(\varphi_{s,t})\cdot \dot{\varphi}_{s,t} + \dot{f}_t(\varphi_{s,t}) =0$ and therefore 
combining to \eqref{evolution} we have the following generalized Loewner-Kufarev PDE
\begin{equation}
\label{generalizeLPDE}
\dot{f}_t(z) = (z - \tau(t))(1-\closure{\tau(t)}z) f_t'(z)p(z,t).
\end{equation}

We shall observe \eqref{generalizeLPDE}.
Since the term $\dot{f}_t(z)$ gives a velocity vector at the point $f_t(z)$, the right-hand side of the equation \eqref{generalizeLPDE} defines a vector field on $f_t(\D)$. 
Assume that $p$ is not identically equal to zero.
Then $\dot{f}_{t}(z) =0$ if $z = \tau(t)$.
It implies that the point $f_t(\tau(t))$ plays a role of an "eye" of the flow described by $f_t(z)$. 
Since the Denjoy-Wolff function $\tau$ is assumed to be only measurable w.r.t. $t$, the origin $f_t(\tau(t))$ of the vector field moves measurably.
This observation indicates that Loewner chain describes various flows of expanding simply connected domains.
The classical radial Loewner-Kufarev PDE is given as the special case of \eqref{generalizeLPDE} with $\tau \equiv 0$.

In general, for a given evolution family $(\varphi_{s,t})$, the equation \eqref{LEtoEV2} does not define a unique Loewner chain.
That is, there is no guarantee that $\LL[(\varphi_{s,t})]$, the family of normalized Loewner chains associated with $(\varphi_{s,t}) \in \EF$, consists of one function.
However, $\LL[(\varphi_{s,t})]$ always includes one special Loewner chain (in \cite{MR2789373}, such a chain is called \textbf{standard}) and in this sense $(f_t)$ is determined uniquely.
Further, it is sometimes the only member of $\LL[(\varphi_{s,t})]$.
The following theorem states such properties of the uniqueness for Loewner chains.

 \begin{thm}[{\cite[Theorem 1.6 and Theorem 1.7]{MR2789373}}]
	
			\label{LCunique}
			
Let $(\varphi_{s,t}) \in \EF$. 
Then there exists a unique normalized $(f_{t}) \in \LC$ such that $\Omega[(f_t)]$ is either $\C$ or an Euclidean disk in $\C$ whose center is the origin.
Furthermore;
\begin{itemize}
\item The following 4 statements are equivalent;
\begin{enumerate}
\item $\Omega[(f_t)]= \C$;
\item $\LL[(\varphi_{s,t})]$ consists of only one function;
\item $\beta (z) = 0$ for all $z \in \D$, where 
$$
\beta (z) := \lim_{t \to +\infty}\frac{|\varphi_{0,t}'(z)|}{1 - |\varphi_{0,t}(z)|^{2}};
$$
\item there exists at least one point $z_{0} \in \D$ such that $\beta(z_{0}) = 0$.\\
\end{enumerate}
\item On the other hand, if $\Omega[(f_t)] \neq \C$, then it is written by 
$$
\Omega[(f_t)] = \left\{ w : |w| < \frac1{\beta(0)}\right\},
$$
and for the other normalized Loewner chain $g_{t}$ associated with $(\varphi_{s,t})$, there exists $h \in \mathcal{S}$ such that
\begin{equation*}
\label{nonstandard}
g_{t}(z) = \frac{h(\beta(0)f_{t}(z))}{\beta(0)}.
\end{equation*}
\end{itemize}
\end{thm}

Here we demonstrate how to construct a normalized Loewner chain $(f_{t}) \in \LC$ from a given evolution family $(\varphi_{s,t}) \in \EF$.
Firstly, define $(\psi_{s,t})_{0 \le s \le t < \infty}$ by
$$
\psi_{s,t} := h_{t}^{-1} \circ \varphi_{s,t} \circ h_{s},
$$
where $h_{t}$ is a M\"obius transformation given by
$$
h_{t}(z) := \frac{b(t)z + a(t)}{1 + \closure{a(t)}b(t)z}, \hspace{15pt} a(t) := \varphi_{0,t}(0), \hspace{15pt} b(t) := \frac{\varphi_{0,t}'(0)}{|\varphi_{0,t}'(0)|}.
$$
Then $(\psi_{s,t}) \in \EF$ (\cite[Proposition 2.9]{MR2789373}). Further it is easy to see that $\psi_{s,t}(0)=0$ and $\psi_{s,t}'(0) >0$ for all $0 \le s \le t < \infty$.
By the $(\psi_{s,t})$, define $(g_{s})_{s \ge 0}$ as
\begin{equation}
\label{gt-limit}
g_{s}(z) := \lim_{t \to \infty} \frac{\psi_{s,t}(z)}{\psi_{0,t}'(0)}.
\end{equation}
Remark that the limit in \eqref{gt-limit} is attained locally uniformly on $\D$.
One can show that $(g_{t}) \in \LC$ associated with $(\psi_{s,t}) \in \EF$ and $g_{0} \in \S$ (\cite[Theorem 3.3]{MR2789373}).
Finally, set
$$
f_{t} := g_{t} \circ h_{t}^{-1}. 
$$
We conclude that $(f_{t})_{t \ge 0} \in \LC$ associated with $(\varphi_{s,t}) \in \EF$ and $f_{0} \in \S$.
In the classical radial case, \eqref{gt-limit} corresponds to \eqref{limitation}.

	%
	%
			\subsection{Quasiconformal extensions for Loewner chains of radial type}
	%

In view of Theorem \ref{Beckerthm}, a natural question is proposed that whether the same assumption for $p \in \HF$ that $p \in U(k)$ deduces quasiconformal extensibility of the corresponding $(f_t) \in \LC$ or not.
We give a positive answer to this problem under the special situation that $\tau \in \DW$ is constant.
According to the case that $\tau \in \D$ or $\tau \in \de\D$, the corresponding setting is called the \textbf{radial case} or \textbf{chordal case}.
In the classical Loewner theory, the first is the original case introduced by L\"owner, and the second is investigated firstly by Kufarev and his students \cite{Kufarev:1968}.

We employ the following definition due to \cite{MR2719792}.

\begin{definition}[{\cite[Definition 1.2]{MR2719792}}]
	
			
Let $(\varphi_{s,t}) \in \EF$.
Suppose that all non-identical elements of $(\varphi_{s,t})$ share the same point $\tau_0 \in \overline{\D}$ such that $\varphi_{s,t}(\tau_0) = \tau_0$ and $|\varphi_{s,t}'(\tau_0)| \le 1$ for all $s \ge 0$ and $t \ge s$, where $\varphi_{s,t}(\tau_0)$ and $ \varphi_{s,t}'(\tau_0)$ are to be understood as the corresponding angular limit if $\tau_0 \in \de \D$.
Then $\varphi_{s,t}$ is said to be a \textbf{radial evolution family} if $\tau_0 \in \D$, or a \textbf{chordal evolution family} if $\tau_0 \in \de \D$.
\end{definition}

Then the radial and chordal version of Loewner chains are defined.

\begin{definition}[{\cite[Definition 1.5]{MR2719792}}]
\label{radial-chordal-LC}
	
			
Let $(f_t) \in \LC$.
If $(\varphi_{s,t})_{0 \le s \le t < \infty} := (f_t^{-1} \circ f_s)_{0 \le s \le t < \infty}$ is a radial (or chordal) evolution family, then we call $(f_t)$ a \textbf{Loewner chain of radial} (or \textbf{chordal}) \textbf{type}.
\end{definition}

Now we prove the following quasiconformal extension criterion for a Loewner chain of radial type.

\begin{thm}
\label{qcextension-general}
Let $k \in [0,1)$ be a constant.
Suppose that $(f_t)$ is a Loewner chain of radial type for which $p \in \HF$ associated with $(f_{t})$ by \eqref{generalizeLPDE}, satisfies 
\begin{equation*}
p(z,t) \in U(k)
\end{equation*}
for all $z \in \D$ and almost all $t\ge 0$ and $\tau \in \DW$ is equal to 0. 
Then the following assertions hold;
\begin{enumerate}
\item $f_t$ admits a continuous extension to $\closure{\D}$ for each $t \geq 0$;
\item $F$ defined in \eqref{beckerequation} gives a $k$-quasiconformal extension of $f_0$ to $\C$;
\item $\Omega[(f_t)] = \C$.
\end{enumerate}
\end{thm}
\begin{proof}
With no loss of generality, we may assume $(f_t) \in \LC$ is normalized, i.e. $f_{0} \in \S$.
Let $\rho \in (c,1)$ with some constant $c \in (0,1)$ and define $f_t^{\rho}(z) := f_t(\rho z)/\rho$. 
Then accordingly $F_{\rho}$ is defined. 
Since $f_{t}^{\rho}(z)$ satisfies $\de_t f_t^{\rho}(z) := z \de_z f_t^{\rho}(z) p(\rho z,t)$, $f_t^{\rho}$ satisfies all the assumptions of our theorem.
Further, $f_t^{\rho}$ is well-defined on $\closure{\D}$ for all $t \ge 0$.

Take two distinct points $z_1, z_2 \in \C$.
If either $z_1$ or $z_2$ is in $\D$, then it is clear that $F_{\rho}(z_1) \neq F_{\rho}(z_2)$.
Suppose $z_1 := r_1 e^{i\theta_1},\,z_2 := r_2 e^{i\theta_2} \in \C\backslash\D$ such that $F_{\rho}(z_1) = F_{\rho}(z_2)$, namely $f_{\log r_1}(\rho e^{i \theta_1}) = f_{\log r_2}(\rho e^{i\theta_2})$.
Denote $t_1 := \log r_1$ and $t_2 := \log r_2$.
Since $f_t^{\rho}(\de \D)$ is a Jordan curve, it follows that $t_1 \neq t_2$.
By the equality condition of the Schwarz lemma we have $\varphi_{t_1, t_2} (z) := f_{t_2}^{-1} \circ f_{t_1} (z) = e^{i\theta}z$ for some $\theta \in \R$.
Hence $p(\D ,t)$ lies on the imaginary axis for all $t \in [t_1, t_2]$ which contradicts our assumption.  
We conclude that $F_{\rho}$ is a homeomorphism on $\C$.

A simple calculation shows that
$$
\left|\frac{\de_{\bar{z}}F_{\rho}(z)}{\de_z F_{\rho}(z)}\right| 
= 
\left|\frac{\de_t f_t^{\rho}(z) - z \de_z f_t^{\rho}(z)}{\de_t f_t^{\rho}(z) + z \de_z f_t^{\rho}(z)}\right|
\leq k
$$
Hence $F_{\rho}$ is $k$-quasiconformal on $\C$.
Since the $k$ does not depend on ${\rho} \in (c,1)$, $(F_{\rho})_{{\rho} \in (c,1)}$ forms a family of $k$-quasiconformal mappings on $\C$ and it is normal.
Therefore the limit $F(z) = \lim_{{\rho} \to 1}F_{\rho}(z)$ exists which gives a $k$-quasiconformal extension of $f_0$.
In particular, $f_t$ is defined on $\de\D$ for all $t \ge 0$.
It also follows from quasiconformality of $F$ that $F(\C) = \Omega[(f_t)] =\C$.
\end{proof}

If $(f_{t})$ is a Loewner chain of radial type and $(p,\tau) \in \BP$ associated with $(f_{t})$ where $p \in U(k)$ and $\tau \in \D\,\backslash\,\{0\}$, $(f_{t})$ satisfies $\dot{f}_{t}(z)=(z-\tau)(1-\bar{\tau}z)f_{t}'(z)p(z,t)$ for all $z\in\D$ and almost all $t \ge 0$.
Let $M$ be a M\"obius transformation defined by
$$
M(z) := \frac{z+\tau}{1+\bar{\tau}z}.
$$
Then $(g_{t})_{t \ge 0} := (f_{t} \circ M)_{t \ge 0}$ is a family of univalent maps satisfying $\dot{g}_{t}(z) = zg_{t}'(z) p(M(z),t)$ for all $z\in\D$ and almost all $t \ge 0$.
By \cite[Theorem 4.1]{MR2789373}, $(g_{t})$ is a Loewner chain whose Berkson-Porta data is $(p,0) \in \BP$.
Applying Theorem \ref{qcextension-general}, $g_{0}$, and hence $f_{0}$, has a $k$-quasiconformal extension to $\C$.

	%
	%
			\subsection{Quasiconformal extensions for Loewner chains of chordal type}
	%

A Loewner chain of chordal type (see Definition \ref{radial-chordal-LC}) with a quasiconformal extension is discussed by Gumenyuk and the author \cite{HottaGumenyuk}.
In the chordal case, $\tau \in \DW$ is a boundary fixed point of $\D$.
By some rotation we may assume that $\tau=1$.

It is sometimes convenient to discuss the chordal case on the not $\D$ but rather the half-plane.
In fact, by means of the conjugation with a Cayley map $K(z)=(1+z)/(1-z)$, everything can be transferred from the unit disk to the right half-plane.
For instance, a family $(\Phi_{s,t})_{0 \le s \le t < \infty}$ of holomorphic self-maps of the right half-plane $\H$ is an \textbf{evolution family} if $(K^{-1} \circ \Phi_{s,t} \circ K)_{0 \le s \le t < \infty}$ is an evolution family on the unit disk $\D$.
Then the generalized chordal Loewner-Kufarev PDE and ODE are written by
\begin{equation}
\label{half-plane}
\dot{\Phi}_{s,t}(\zeta) = p_{\H}(\Phi_{s,t} (\zeta), t) 
\hspace{20pt}
\textup{and}
\hspace{20pt}
\dot{f}_t(\zeta) = -f_t'(\zeta) p_{\H}(\zeta, t)\hspace{15pt}(\zeta \in \H),
\end{equation}
where $p_{\H}(\zeta,t) := 2p(K^{-1}(\zeta), t)$ stands for the right half-plane version of the Herglotz function.
A special case of \eqref{half-plane} that $p_{\H}(\zeta) = 1/(\zeta + i\lambda(t))$ has attracted great attention since the work by Schramm \cite{Schramm:2000} was provided, where $\lambda : [0,\infty) \to \R$ is a measurable function.

The next theorem states the chordal variant of Becker's theorem.

\begin{thm}[\cite{HottaGumenyuk}]
	
			\label{mainthm01}
			
Suppose that a family of holomorphic functions $(f_t)_{t\ge 0}$ on the right half-plane $\H$ is a Loewner chain of chordal type.
If there exists a uniform constant $k \in [0,1)$ such that $p_{\H}$, a Herglotz function associated with $(f_{t})$, satisfies
\begin{equation}
	p_{\H}(\zeta,t) \in U(k)
\end{equation}
for all $\zeta \in \H$ and almost all $t \geq 0$, then 
\begin{enumerate}
\def\labelenumi{(\roman{enumi})}
\item $f_t$ admits a continuous extension to $\H \cup i\R$;
\item  $f_t$ has a $k$-quasiconformal extension to $\C$ for each $t \ge 0$.
In this case the extension $F$ is explicitly given by
$$
F(\zeta) := \left\{
\begin{array}{ll}
f_0(\zeta), & \zeta \in \H,\\
\displaystyle f_{ - \Re \zeta}( i\, \Im \zeta), & \zeta \in \C\backslash\overline{\H};
\end{array}
\right.
$$
\item $\Omega[(f_t)] = \C$.
\end{enumerate}
\end{thm}

If $\tau \in \DW$ is a boundary point on $\de\D\,\backslash\,\{1\}$, then composing a proper rotation we obtain the same result as Theorem \ref{mainthm01}.
In fact, by setting $g_t(z) := f_t(\bar{\tau}z)$ we have $g_t(z) = (z-\tau)(1-\bar{\tau}z)g'(z) p(\bar{\tau}z, t)$.
After transferring $g_t$ to the right half-plane, Theorem \ref{mainthm01} with the same $k$ as $f_t$ is applied.

\

	%
	%

\bibliographystyle{amsalpha}
\def\cprime{$'$}
\providecommand{\bysame}{\leavevmode\hbox to3em{\hrulefill}\thinspace}
\providecommand{\MR}{\relax\ifhmode\unskip\space\fi MR }
\providecommand{\MRhref}[2]{%
  \href{http://www.ams.org/mathscinet-getitem?mr=#1}{#2}
}
\providecommand{\href}[2]{#2}

\end{document}